\documentclass[11pt,letterpaper]{amsart}
\pdfoutput=1
  
\usepackage{amssymb}
\usepackage{mathtools}
\usepackage{microtype}
\usepackage[shortlabels]{enumitem}
\usepackage[bookmarks, bookmarksdepth=2, colorlinks=true, linkcolor=blue, citecolor=blue, urlcolor=blue]{hyperref}

\setenumerate[0]{label=(\alph*)}

\newcommand{\arxiv}[1]{\href{http://arxiv.org/abs/#1}{{\tt arXiv:#1}}}

\numberwithin{equation}{section}

\theoremstyle{plain}
\newtheorem{theorem}{Theorem}[section]
\newtheorem{maintheorem}{Theorem}

\newtheorem{proposition}[theorem]{Proposition}
\newtheorem{lemma}[theorem]{Lemma}

\newtheorem{corollary}[theorem]{Corollary}

\theoremstyle{definition}

\newtheorem{defn}[theorem]{Definition}
\newenvironment{definition}[1][]{\begin{defn}[#1]\pushQED{\qed}}{\popQED \end{defn}}

\theoremstyle{remark}
\newtheorem{rmk}[theorem]{Remark}
\newenvironment{remark}[1][]{\begin{rmk}[#1] \pushQED{\qed}}{\popQED \end{rmk}}

\newtheorem{eg}[theorem]{Example}
\newenvironment{example}[1][]{\begin{eg}[#1] \pushQED{\qed}}{\popQED \end{eg}}

\DeclareMathOperator{\Lie}{Lie}
\DeclareMathOperator{\Spec}{Spec}

\DeclareMathOperator{\GL}{GL}

\newcommand\R{\ensuremath{\mathbb{R}}}
\newcommand\C{\ensuremath{\mathbb{C}}}
\newcommand\Z{\ensuremath{\mathbb{Z}}}
\newcommand\Q{\ensuremath{\mathbb{Q}}}

\newcommand\Field{\ensuremath{\mathbb{F}}}

\DeclareMathOperator{\HH}{H}
\newcommand\RH{\ensuremath{\widetilde{\HH}}}
\DeclareMathOperator{\CC}{C}
\newcommand\RC{\ensuremath{\widetilde{\CC}}}

\DeclareMathOperator{\Char}{char}

\newcommand\Set[2]{\ensuremath{\left\{\text{#1 $|$ #2}\right\}}}

\newcommand\cA{\ensuremath{\mathcal{A}}}

\newcommand\cC{\ensuremath{\mathcal{C}}}

\newcommand\cF{\ensuremath{\mathcal{F}}}

\newcommand\cO{\ensuremath{\mathcal{O}}}

\newcommand\cW{\ensuremath{\mathcal{W}}}

\newcommand\fa{\ensuremath{\mathfrak{a}}}
\newcommand\fb{\ensuremath{\mathfrak{b}}}
\newcommand\fc{\ensuremath{\mathfrak{c}}}
\newcommand\fd{\ensuremath{\mathfrak{d}}}

\newcommand\fu{\ensuremath{\mathfrak{u}}}

\newcommand\bA{\ensuremath{\mathbf{A}}}
\newcommand\bB{\ensuremath{\mathbf{B}}}

\newcommand\bG{\ensuremath{\mathbf{G}}}

\newcommand\bT{\ensuremath{\mathbf{T}}}
\newcommand\bU{\ensuremath{\mathbf{U}}}
\newcommand\bV{\ensuremath{\mathbf{V}}}

\newcommand\bX{\ensuremath{\mathbf{X}}}

\newcommand\bbA{\ensuremath{\mathbb{A}}}

\newcommand\bbF{\ensuremath{\mathbb{F}}}
\newcommand\bbG{\ensuremath{\mathbb{G}}}

\newcommand\bbL{\ensuremath{\mathbb{L}}}

\newcommand\bbP{\ensuremath{\mathbb{P}}}

\newcommand\bbZ{\ensuremath{\mathbb{Z}}}

\newcommand\oS{\ensuremath{\overline{S}}}

\newcommand\ok{\ensuremath{\overline{k}}}

\DeclareMathOperator{\diag}{diag}
\DeclareMathOperator{\St}{St}
\DeclareMathOperator{\Tits}{\ensuremath{\mathcal{T}}}
\DeclareMathOperator{\minPara}{\mathfrak{P}_{\min}}
\DeclareMathOperator{\apart}{\mathfrak{A}}

\DeclareMathOperator{\RelRoot}{\Phi_{\mathrm{rel}}}

\DeclareMathOperator{\RelRootPlus}{\Phi^+_{\mathrm{rel}}}

\DeclareMathOperator{\OneParam}{\sigma}
\newcommand{\id}{\mathrm{id}}
\newcommand\presup[2]{\prescript{#1}{}{#2}}

\title{The Steinberg representation is irreducible}

\author{Andrew Putman}
\address{Department of Mathematics; University of Notre Dame; 255 Hurley Hall; Notre Dame, IN 46556}
\email{andyp@nd.edu}
\thanks{AP was supported in part by NSF grant DMS-1811210.}

\author{Andrew Snowden}
\address{Department of Mathematics; University of Michigan; 2074 East Hall; 530 Church St; Ann Arbor MI 48109}
\email{asnowden@umich.edu}
\thanks{AS was supported in part by NSF grant DMS-1453893}

\date{February 17, 2022}

\begin{document}

\newpage

\begin{abstract}
We prove that the Steinberg representation of a connected reductive group over an infinite field is
irreducible.  For finite fields, this is a classical theorem of Steinberg and Curtis.
\end{abstract}

\maketitle
\tableofcontents

\section{Introduction}
\label{section:introduction}

Let $\bG$ be a connected reductive group over a field $k$, e.g., $\bG = \GL_n$.  For another field $\bbF$, let $\St(\bG;\bbF)$ be the Steinberg representation  of the discrete group\footnote{The Steinberg
representation is not an algebraic representation of $\bG$, but just a representation of
the abstract group $\bG(k)$.  However, its definition uses the structure of $\bG$ as an
algebraic group defined over $k$, so it would not make sense to write it as
$\St(\bG(k);\bbF)$.}
 $\bG(k)$ over $\bbF$. This representation plays a prominent role in the representation theory of $\bG(k)$, and also has connections to number theory and K-theory. When $k$ is finite, $\St(\bG;\bbF)$ is finite dimensional and Steinberg and Curtis showed that it is usually
irreducible. When $k$ is infinite, it is typically infinite dimensional. Our main theorem is that
$\St(\bG;\bbF)$ is always irreducible when $k$ is infinite.  Previously, this was not known in complete generality
even for $\bG=\GL_2$.

\subsection{Background}

Before explaining the contents of this paper in more detail, we recall the construction of the Steinberg representation, review some of its history, and discuss its connections to other topics.

\subsubsection{Tits building}

A tremendous amount of the structure of $\bG$ is encoded in its spherical
Tits building $\Tits(\bG)$.  This is a simplicial 
complex whose simplices are in
bijection with the proper parabolic $k$-subgroups of $\bG$.  For proper parabolic $k$-subgroups
$P$ and $P'$, the simplex corresponding to $P$ is a face of the simplex corresponding to $P'$
when $P' \subset P$.  The conjugation action of $\bG(k)$ on itself permutes the parabolic $k$-subgroups and thus induces an action
of $\bG(k)$ on $\Tits(\bG)$.  See \cite{AbramenkoBrownBuildings, TitsBuildings} for more about Tits buildings.

\begin{example}
If $\bG = \GL_n$, then the proper parabolic $k$-subgroups
are the stabilizers of nontrivial flags
\[0 \subsetneq V_0 \subsetneq V_1 \subsetneq \cdots \subsetneq V_i \subsetneq k^n,\]
so $\Tits(\GL_n)$ is the simplicial complex whose $i$-simplices are such flags.  
\end{example}

\begin{remark}
It is not obvious that the above description of $\Tits(\bG)$ specifies a simplicial complex. However,
what we really care about is the homology of $\Tits(\bG)$, and for this it is enough to
understand its barycentric subdivision, which is easy to describe completely:
it is the simplicial complex whose $i$-simplices are decreasing chains
\[\bG \supsetneq P_0 \supsetneq P_{1} \supsetneq \cdots \supsetneq P_i \supsetneq 1\]
of proper parabolic $k$-subgroups.
\end{remark}

\subsubsection{Steinberg representation}
Let $r$ be the semisimple $k$-rank of $\bG$, e.g., if $\bG = \GL_n$ then $r = n-1$.  By definition, $\Tits(\bG)$ is an $(r-1)$-dimensional
simplicial complex.
The Solomon--Tits theorem \cite{SolomonTits} says that
$\Tits(\bG)$ is homotopy equivalent to a wedge of $(r-1)$-dimensional spheres.  For a field
$\bbF$ (or, more generally, a commutative ring), the {\em Steinberg representation} of $\bG(k)$ over $\bbF$, denoted $\St(\bG;\bbF)$,
is the unique nontrivial reduced homology group $\RH_{r-1}(\Tits(\bG);\bbF)$.  The action
of $\bG(k)$ on $\Tits(\bG)$ induces an action of $\bG(k)$ on $\St(\bG;\bbF)$, making it into
a representation of $\bG(k)$ over $\bbF$.

\begin{remark}
\label{remark:anisotropic}
It might be the case that $\bG$ is anisotropic, i.e., has no proper parabolic $k$-subgroups.  This
implies that $\Tits(\bG) = \emptyset$ and that the semisimple $k$-rank of $\bG$ is $0$.  Our convention
then is that
\[\St(\bG;\bbF) = \RH_{-1}(\Tits(\bG);\bbF) = \RH_{-1}(\emptyset;\bbF) = \bbF\]
is the trivial representation. 
\end{remark}

\subsubsection{Finite fields}
The representation $\St(\bG;\bbF)$ was first studied for finite fields $k$.  In this case,
$\St(\bG;\bbF)$ is a finite-dimensional representation of the finite group
$\bG(k)$ that is 
usually\footnote{There are cases where it is reducible.  For example, it is reducible if $\bG = \GL_2$, the field
$k$ is finite of cardinality $q$, and $\Field$ has finite characteristic $\ell$ with $\ell \mid q+1$.} irreducible.
For instance,
this holds if $\Char(\bbF) = 0$ or if $\Char(\bbF) = \Char(k)$. 
Steinberg \cite{SteinbergFirst} initially proved this for $\bG = \GL_n$, and then
generalized it to many other finite groups \cite{SteinbergSecond, SteinbergThird}.
Curtis \cite{CurtisBN} proved the ultimate version for a finite group with a BN-pair.
See \cite{HumphreysSurvey, SteinbergSurvey} for surveys of the fundamental 
role the Steinberg representation plays in the representation theory of finite groups
of Lie type.

\begin{remark}
The above papers predate the definition of the Steinberg representation in terms of the Tits building,
which first appeared in \cite{SolomonTits}.
\end{remark}

\subsubsection{Infinite fields}
For infinite fields $k$, the representation $\St(\bG;\bbF)$
is usually\footnote{The only time when $\St(\bG;\bbF)$ is finite-dimensional
for infinite $k$ is when $\bG$ is anisotropic, in which case $\St(\bG;\bbF) = \bbF$ is the trivial representation.} 
an infinite-dimensional representation
 of the infinite group $\bG(k)$.
In this context, it first appeared in work of Borel--Serre \cite{BorelSerreCorners}, who proved that
for algebraic number fields $k$
the symmetric space associated to the Lie group $\bG(k \otimes_{\Q} \R)$ has a $\bG(k)$-equivariant bordification
whose boundary is homotopy equivalent to $\Tits(\bG)$.  They used this to show that $\St(\bG;\bbF)$ is the ``dualizing module''
for arithmetic subgroups of $\bG(k)$.  This gave rise to a large literature using $\St(\bG;\bbF)$ to study
the cohomology of such arithmetic subgroups.  Some 
representative papers include 
\cite{AshGunnellsMcConnell, AshRudolph, Bykovskii, ChurchFarbPutman, ChurchPutmanCodimOne,
GunnellsSymplectic, KupersMillerPatztWilson, LeeSzczarba, MillerNagpalPatzt, MillerPatztPutman,
MillerPatztWilsonYakasi, PutmanStudenmund}.

A second important context for $\St(\bG;\bbF)$ when $k$ is infinite is algebraic K-theory, where
Quillen \cite{QuillenFG} constructed a spectral sequence converging to the algebraic K-theory of a number ring
$\cO$ whose $E^2$-page involves the homology of $\GL_n(\cO)$ with coefficients in the
Steinberg representation.  His main application was to show that these K-groups are finitely generated.
The Steinberg representation and related objects have since appeared in a variety of K-theoretic and
homotopy theoretic contexts.
Some representative papers include \cite{
AroneDwyer,
GalatiusKupersRandalWilliams1,
GalatiusKupersRandalWilliams2,
MitchellPriddy,
RognesRank,
RognesK4,
SikiricElbazKupersMartinet,
SouleK4}

\subsubsection{Irreducibility for infinite fields}

It is natural to wonder whether $\St(\bG;\bbF)$ is irreducible when $k$ is infinite.
This was first studied by Xi \cite{XiSteinberg}, who proved
that $\St(\bG;\bbF)$ is irreducible when $\bG$ is defined over the algebraic
closure $k=\overline{\bbF}_q$ of a finite field $\bbF_q$ and $\bbF$ is a field with
$\Char(\bbF) \in \{0, \Char(k)\}$.  Yang \cite{YangSteinberg} later removed this
restriction on $\Char(\bbF)$.  Both proofs make essential use of
the fact that $k$ is a union of finite fields, and do not appear to
generalize to more general fields $k$.
More recently, Galatius--Kupers--Randal-Williams \cite{GalatiusKupersRandalWilliams2}
proved that $\St(\GL_n;\bbF)$ is an indecomposable\footnote{An indecomposable representation is one that cannot
be decomposed as a nontrivial direct sum of two subrepresentations.  When $k$ is infinite, the Steinberg representation
is typically infinite-dimensional and this is a weaker condition than being irreducible even when $\Field$ has characteristic
$0$.} representation of $\GL_n(k)$ for all fields $k$ and $\Field$.

\subsection{Main theorem}

Our main theorem answers this question completely.

\begin{maintheorem}
\label{maintheorem:irreducible}
Let $\bG$ be a connected reductive group over an infinite field $k$ and let $\bbF$ be an arbitrary
field. Then the Steinberg representation $\St(\bG;\bbF)$ is an irreducible $\bG(k)$-representation.
\end{maintheorem}

\begin{remark}
For a local field $k$, there is a variant of the Steinberg representation
that takes into account the topology of $k$ (see, e.g., \cite{BorelSerreSArithmetic}).
This variant is irreducible (see \cite[Example~7.6]{PrasadRaghuram} or \cite[Example~9.2]{Zelevinsky}), but it is different 
enough from the ordinary Steinberg representation that this does not seem to
imply Theorem~\ref{maintheorem:irreducible} when $\bG$ is defined over a local field.
\end{remark}

\begin{remark}
There are a number of other representations that are similar to the Steinberg representation. For example, when $k$ is finite and $\bbF$ has characteristic~0, there is one irreducible unipotent representation of $\GL_n(k)$ for each partition of $n$, with the Steinberg representation corresponding to the partition $(1^n)$. Our methods should be able to prove irreducibility for these analogous representations, though we have not pursued this.
\end{remark}

\subsection{The case where \texorpdfstring{$k$}{k} is finite}

In most cases, our proof uses the fact that $k$ is infinite in a serious way.  Since
$\St(\bG;\bbF)$ is sometimes reducible when $k$ is finite, this is inevitable.  However,
our work does give the most important special cases of irreducibility for finite $k$.
Let $k$ be a finite field with $\Char(k) = p$.
\begin{itemize}
\item When $\Char(\bbF) = p$, our proof works in complete
generality and many aspects of it simplify.\footnote{In particular, the second
ingredient (Proposition \ref{prop:ideal}) in the proof outline discussed in
\S \ref{section:outline} is almost trivial in this case; see the
very short \S \ref{section:charpequal} for details.}
\item The case where $\Char(\bbF) = 0$ follows from the
case where $\Char(\bbF) = p$.  To see this, observe that
$\St(\bG;\Z)$ is a finite-rank free abelian group with
$\St(\bG;\overline{\bbF}_p) = \St(\bG;\Z) \otimes \overline{\bbF}_p$
and $\St(\bG;\bbF) = \St(\bG;\Z) \otimes \bbF$.  We now quote the following standard
result:\footnote{Here is a quick proof.  If $M \otimes \bbF$ is reducible
for some $\bbF$ with $\Char(\bbF) = 0$, then letting $\overline{\bbF}$ be an
algebraic closure of $\bbF$ we have that $M \otimes \overline{\bbF}$ is reducible.  Since
they have the same characters, this
implies that $M \otimes \overline{\bbL}$ is reducible for any algebraically closed field
$\overline{\bbL}$ with $\Char(\overline{\bbL}) = 0$; in 
particular, $M \otimes \overline{\Q}_p$ is reducible.  Let
$V \subset M \otimes \overline{\Q}_p$ be a nonzero proper subrepresentation.  Letting $\overline{\Z}_p$ be the
ring of integers in $\overline{\Q}_p$, the intersection $V \cap (M \otimes \overline{\Z}_p)$ is a nonzero proper direct
summand of the $\overline{\Z}_p$-module $M \otimes \overline{\Z}_p$, and hence maps to a nonzero proper subrepresentation of $M \otimes \overline{\bbF}_p$
under the reduction map $M \otimes \overline{\Z}_p \rightarrow M \otimes \overline{\bbF}_p$.}
if $G$ is a finite group,
$M$ is a $\Z[G]$-module whose underlying abelian group is finite-rank and free, and
$M \otimes \overline{\bbF}_p$ is irreducible, then $M \otimes \bbF$ is irreducible for all
fields $\bbF$ with $\Char(\bbF) = 0$.
\end{itemize}

\subsection{Outline of proof of Theorem~\ref{maintheorem:irreducible}}
\label{section:outline}

Let the notation be as in Theorem~\ref{maintheorem:irreducible}.
Let $\bB$ be a minimal parabolic $k$-subgroup of $\bG$, let $\bU$ be the unipotent radical of $\bB$, and let 
$\bT$ be a maximal $k$-split torus in $\bB$. For example, if $\bG=\GL_n$ then one can take 
$\bB$ to be the Borel subgroup of upper triangular matrices, $\bU$ to be group of upper triangular matrices with $1$'s on the diagonal, 
and $\bT$ to be the group of diagonal matrices. 

A strengthening of the Solomon--Tits theorem gives a linear isomorphism 
$\iota \colon \St(\bG; \bbF) \to \bbF[\bU(k)]$.  The map $\iota$ is equivariant for $\bU(k)$ 
and $\bT(k)$, which act by left multiplication and conjugation on the target, respectively.
However, the action of a general element of $\bG(k)$ on $\bbF[\bU(k)]$ is opaque.
The actions of $\bU(k)$ and $\bT(k)$ on $\bbF[\bU(k)]$ preserve the augmentation ideal, i.e., the kernel of the augmentation $\epsilon\colon \bbF[\bU(k)] \rightarrow \bbF$.
We will prove the following:

\begin{proposition}
\label{prop:augment}
Let the notation be as above and let $x \in \St(\bG;\bbF)$ be non-zero.  Then there exists $g \in \bG(k)$ such that $\iota(gx)$ is
not in the augmentation ideal.
\end{proposition}

\begin{proposition}
\label{prop:ideal}
Let the notation be as above and
let $I \subset \bbF[\bU(k)]$ be a left ideal that is stable under $\bT(k)$ and 
not contained in the augmentation ideal. Then $I=\bbF[\bU(k)]$.
\end{proposition}

To deduce Theorem~\ref{maintheorem:irreducible}, let $V \subset \St(\bG;\bbF)$ be a non-zero subrepresentation.
Then $\iota(V)$ is a $\bT(k)$-stable left ideal of $\bbF[\bU(k)]$ that by Proposition~\ref{prop:augment} is
not contained in the augmentation ideal.  Proposition~\ref{prop:ideal} thus implies that $\iota(V)=\bbF[\bU(k)]$, 
so $V=\St(\bG; \bbF)$.

\subsection{Special case of Proposition~\ref{prop:augment}}

To prove Proposition \ref{prop:augment}, we must relate the augmentation map $\epsilon \colon \bbF[\bU(k)] \to \bbF$
to the structure of the Tits building $\Tits(\bG)$.  Doing this in general requires introducing a lot
of building-theoretic terminology (chambers, apartments, etc.).  To give the basic idea, we will
explain how this works for $\bG = \GL_2$.

\subsubsection{Structure of building}
We must first construct $\iota$.  The parabolic subgroups of $\GL_2$ are the stabilizers of lines in $k^2$.  The
Tits building $\Tits(\GL_2)$ can thus be identified with the discrete set $\bbP^1(k)$.
Elements of $\HH_0(\Tits(\GL_2);\bbF)$ are
formal $\bbF$-linear combinations of points of $\bbP^1(k)$.  The Steinberg representation is
the reduced homology:
\[\St(\GL_2;\bbF) = \RH_0(\Tits(\GL_2);\bbF) = \Set{$\sum_{i=0}^n c_i \ell_i$}{$\ell_i \in \bbP^1(k)$, $c_i \in \bbF$, $\sum_{i=0}^n c_i = 0$}.\]

\subsubsection{Apartment classes}
The Steinberg representation is spanned by elements of the form $\ell - \ell'$ for distinct $\ell,\ell' \in \bbP^1(k)$,
which are called {\em apartment classes}.  These are not linearly independent.  Using homogeneous
coordinates on $\bbP^1(k)$, the apartment classes of the form $[1,0]-[\lambda,1]$ are a basis
for $\St(\GL_2;\bbF)$.
Let $\bU$ be the unipotent subgroup of upper triangular $2 \times 2$ matrices with $1$'s on the diagonal. We thus have an isomorphism
\[\iota\colon \St(\GL_2;\bbF) \rightarrow \bbF[\bU(k)], \qquad \iota\left([1,0] - [\lambda,1]\right) = \left(\begin{matrix} 1 & \lambda \\ 0 & 1 \end{matrix}\right).\]

\subsubsection{Making the augmentation nonzero}
Let $\epsilon\colon \bbF[\bU(k)] \rightarrow \bbF$ be the augmentation. Consider a nonzero $x \in \St(\GL_2;\bbF)$.  
We must find some $g \in \GL_2(k)$ with $\epsilon(\iota(gx)) \neq 0$.  Write
\[x = \sum_{i=0}^n c_i \ell_i \quad \text{with $\ell_i \in \bbP^1(k)$, $c_i \in \bbF$, and $\sum_{i=0}^n c_i = 0$}\]
with the $\ell_i$ all distinct and the $c_i$ all nonzero.  Pick 
$g \in \GL_2(k)$ with $g \ell_0 = [1,0]$.  For $1 \leq i \leq n$, write
$g \ell_i = [\lambda_i,1]$ with $\lambda_i \in k$.  Since $\sum_{i=0}^n c_i = 0$, it follows that
\[gx = c_0 [1,0] + \sum_{i=1}^n c_i [\lambda_i,1] = \sum_{i=1}^n -c_i \left([1,0] - [\lambda_i,1]\right),\]
so 
\[\epsilon(\iota(gx)) = \epsilon\left(\sum_{i=1}^n -c_i \left(\begin{matrix} 1 & \lambda_i \\ 0 & 1 \end{matrix}\right)\right) = \sum_{i=1}^n -c_i = c_0 \neq 0.\qedhere\]

\subsection{Special case of Proposition~\ref{prop:ideal}} 
\label{section:specialideal}

We now explain our proof of Proposition~\ref{prop:ideal} for $\bG = \GL_2$ when $\Char(k)=0$.
Here $\bU$ is the unipotent group of upper triangular $2 \times 2$ matrices with $1$'s on the diagonal and
$\bT$ is the group of $2 \times 2$ diagonal matrices.
Let $I \subset \bbF[\bU(k)]$ be a left ideal that is stable under the conjugation action of $\bT(k)$ and not contained in
the augmentation ideal.  We must prove that $I = \bbF[\bU(k)]$.

\subsubsection{Torus action}
Since $I \not\subset \ker(\epsilon)$, we can find $x \in I$ with $\epsilon(x) = 1$.  Write this as
\[x = \sum_{i=1}^n c_i \left(\begin{matrix} 1 & \lambda_i \\ 0 & 1\end{matrix}\right) \in I \quad \text{with $\lambda_1,\ldots,\lambda_n \in k$, $c_1,\ldots,c_n \in \bbF$, and $\sum_{i=1}^n c_i = 1$}.\]
Since $\Char(k) = 0$, the torus $\bT(k)$ contains matrices $\diag(d,1)$ for all nonzero $d \in \Z$.  We have 
\[\left(\begin{matrix} d & 0 \\ 0 & 1 \end{matrix}\right) 
\left(\begin{matrix} 1 & \lambda \\ 0 & 1 \end{matrix}\right)
\left(\begin{matrix} d & 0 \\ 0 & 1 \end{matrix}\right)^{-1}
= \left(\begin{matrix} 1 & d \lambda \\ 0 & 1 \end{matrix}\right) \quad \text{for all $\lambda \in k$}.\]
Letting $x_d \in I$ be the result of conjugating $x \in I$ by $\diag(d,1)$, we thus have
\begin{equation}
\label{eqn:writexn}
x_d = \sum_{i=1}^n c_i \left(\begin{matrix} 1 & d \lambda_i \\ 0 & 1\end{matrix}\right) \in I \quad \text{for all nonzero $d \in \Z$}.
\end{equation}

\subsubsection{Laurent polynomials}
Let $\Psi \colon \bbZ[z_1^{\pm 1},\ldots,z_n^{\pm 1}] \rightarrow \bbF[\bU(k)]$ be
the ring homomorphism defined via the formula
\begin{equation}
\label{eqn:psidefinition}
\Psi(z_1^{d_1} \cdots z_n^{d_n}) = \left(\begin{matrix} 1 & d_1 \lambda_1 + \cdots + d_n \lambda_n \\ 0 & 1 \end{matrix}\right).
\end{equation}
Use $\Psi$ to make $\bbF[\bU(k)]$ into a left module over $\bbZ[z_1^{\pm 1}, \ldots, z_n^{\pm 1}]$: for
$f \in \bbZ[z_1^{\pm 1}, \ldots, z_n^{\pm 1}]$ and $x \in \bbF[\bU(k)]$, define $f \cdot x = \Psi(f) x$.
Letting $\id \in \bU(k)$ be the identity matrix, we then have
\[x_d = \sum_{i=1}^n z_i^d \cdot (c_i \cdot \id) \in I \quad \text{for all nonzero $d \in \Z$}.\]
To prove the proposition, it is enough to show that $\id=\sum_{i=1}^n c_i \cdot \id \in I$. 

\subsubsection{Modules over Laurent polynomials}
For this, apply the following lemma with
\[M = \bbF[\bU(k)] \quad \text{and} \quad N = I \quad \text{and} \quad m_i = c_i \cdot \id .\]

\begin{lemma}
\label{lemma:polylemma}
Let $R =\bbZ[z_1^{\pm 1},\ldots,z_n^{\pm 1}]$, let $M$ be an $R$-module, let $N \subset M$ be a submodule, and let $m_1,\ldots,m_n \in M$. Assume that
$z_1^d \cdot m_1 + \cdots + z_n^d \cdot m_n \in N$ for all $d \geq 1$.  Then $m_1+\cdots+m_n \in N$.
\end{lemma}
\begin{proof}
Replacing $M$ by $M/N$, we can assume that $N=0$.  Also, replacing $M$ by the $R$-span of the $m_i$, we can assume that $M$ is finitely generated. Let $\fa$ be a maximal ideal of $R$ and let $r \ge 1$. Since $R/\fa^r$ is a finite ring and each $z_i$ is a multiplicative unit in $R$, we can find some $d \geq 1$ such that $z_i^d \equiv 1 \pmod{\fa^r}$ for all $1 \leq i \leq d$.  We then have
\[0 = z_1^d \cdot m_1 + \cdots + z_n^d \cdot m_n \equiv m_1 + \cdots + m_n \pmod{\fa^r},\]
so $m_1+\cdots+m_n \in \fa^r M$. Since this holds for all $r$, we see that $m_1+\cdots+m_n \in \bigcap_{r \ge 1} \fa^r M$, so by the
Krull intersection theorem\footnote{What the reference \cite[Corollary 5.4]{Eisenbud} actually proves
is that there is some $z \in \fa$ such that $1-z$ annihilates $\cap_{r \geq 1} \fa^r M$.  Since
$1-z$ is invertible in the localization $R_{\alpha}$, it follows that $\cap_{r \geq 1} \fa^r M$
maps to $0$ in $M_{\alpha}$.}
 \cite[Corollary 5.4]{Eisenbud} the element $m_1+\cdots+m_r$ maps to $0$ in the localization $M_{\fa}$.
Since this holds for all maximal ideals $\fa$, it follows that $m_1+\cdots+m_r=0$, as required.
\end{proof}

\begin{remark}
The following special case of Lemma \ref{lemma:polylemma} might clarify its content.
Consider $a_1,\ldots,a_n \in \C^{\times}$ and $b_1,\ldots,b_n \in \C$, and assume that
$a_1^d b_1 + \cdots + a_n^d b_n = 0$ for all $d \geq 1$.  Then Lemma \ref{lemma:polylemma}
implies that $b_1+\cdots+b_n = 0$.  
It is curious that we proved such a simple statement by reduction to finite characteristic\footnote{Of course, this statement can be proven directly. However, the proof of the lemma shows that the conclusion still holds if the stated condition only holds for all $d$ in a cofinal subset of $\bbZ$ (ordered by divisibility). This stronger statement is not so easy to prove by hand.}. When $\Char(k)=0$, our proof 
of Proposition~\ref{prop:ideal} makes use of a similar reduction to finite characteristic.  However,
though this might lead one to expect that the proof when $\Char(k) = p$ would be easier, in fact
our argument in characteristic $p$ is completely different and significantly more technical.
\end{remark}

\begin{remark}
\label{remark:concrete}
Our proof of Lemma \ref{lemma:polylemma} is a little abstract.  It is an instructive exercise to prove it more concretely by exhibiting
appropriate polynomial identities.  For instance, the case $n=3$ follows from the identity
\begin{align*}
z_1 z_2 z_3 (m_1+m_2+m_3) = &(z_1 z_2 + z_1 z_3 + z_2 z_3)(z_1 m_1 + z_2 m_2 + z_3 m_3) \\
                              &-(z_1+z_2+z_3)(z_1^2 m_1 + z_2^2 m_2 + z_3^2 m_3) \\
                              &+(z_1^3 m_1 + z_2^3 m_2 + z_3^3 m_3).
\end{align*}
Since the right hand side lies in $N$, the left hand side does as well. As the $z_i$ are units in $R$, we have $m_1+m_2+m_3 \in N$.
\end{remark}

\subsection{Comments on general case of Proposition~\ref{prop:ideal}}     
\label{section:generalideal}

We close the introduction by saying a few words about the general case of Proposition~\ref{prop:ideal}.  We actually
prove a more general result that applies to arbitrary unipotent groups.

\subsubsection{Key identity}
To state this, we must abstract the necessary properties of the action of $\bT$ on $\bU$.
The key identity that powered our proof of Proposition~\ref{prop:ideal} when $\bG = \GL_2$ and $\Char(k)=0$ is
\[\left(\begin{matrix} d & 0 \\ 0 & 1 \end{matrix}\right) 
\left(\begin{matrix} 1 & \lambda \\ 0 & 1 \end{matrix}\right)
\left(\begin{matrix} d & 0 \\ 0 & 1 \end{matrix}\right)^{-1}
= \left(\begin{matrix} 1 & d \lambda \\ 0 & 1 \end{matrix}\right) \quad \text{for nonzero $d \in \Z$}.\]
We could also have used the matrices $\diag(1,d^{-1})$, for which the analogous formula
is
\[\left(\begin{matrix} 1 & 0 \\ 0 & d^{-1} \end{matrix}\right) 
\left(\begin{matrix} 1 & \lambda \\ 0 & 1 \end{matrix}\right)
\left(\begin{matrix} 1 & 0 \\ 0 & d^{-1} \end{matrix}\right)^{-1}
= \left(\begin{matrix} 1 & d \lambda \\ 0 & 1 \end{matrix}\right) \quad \text{for nonzero $d \in \Z$}.\]
The choice to use $d = d^1$ or $d^{-1}$ reflects the
fact that the weights\footnote{These diagonal subgroups are each isomorphic to the multiplicative
group $\bbG_m = \GL_1$.  Recall (\cite[\S III.8.17]{BorelBook} or \cite[IV.1.1.6]{DemazureGabriel}) that if $\bbG_m$ acts algebraically and linearly on a vector space $V$, then $V$ decomposes as a direct
sum of weight spaces $V_d$, where $\bbG_m(k) = k^{\times}$ acts on $V_d$ as $t \cdot v = t^d v$ for
$t \in k^{\times}$ and $v \in V_d$.  The integers $d$ with $V_d \neq 0$ are called the {\em weights} of the action.  A similar result holds for other diagonalizable groups (e.g., the split tori $\left(\bbG_m\right)^{\times r})$), but with the integral weights replaced by characters.} of the actions of $\diag(\ast,1)$ and $\diag(1,\ast)$ on the Lie algebra $\Lie(\bU)$ of $\bU$ are $1$ and $-1$, respectively.

\subsubsection{Positive actions}
When $\dim(\bU)>1$, there may be more than one such weight.  If we try to imitate the above
proof, it turns out that we will run into trouble if there are both positive and negative weights (roughly, we
won't be able make a single ``choice of $d^1$ or $d^{-1}$'').  Composing a $\bbG_m$-action on $\bU$
with the inversion involution on $\bbG_m$ changes the signs of the weights, so we might
as well assume they are all positive as in the following: 

\begin{definition}
An action of $\bbG_m$ on a smooth connected unipotent group $\bU$ over a field $k$ is said to
be {\em positive} if the weights of the induced action of $\bbG_m$ on the Lie algebra $\Lie(\bU)$ of $\bU$ are positive.
\end{definition}

We will prove the following.

\begin{maintheorem}
\label{maintheorem:positiveideal}
Let $\bU$ be a smooth connected unipotent group over an infinite field $k$ equipped with a positive action of $\bbG_m$ and let
$\bbF$ be another field.  Let $I \subset \bbF[\bU(k)]$ be a left ideal that is stable under $\bbG_m$ and not contained in the augmentation
ideal.  Then $I = \bbF[\bU(k)]$.
\end{maintheorem}

\noindent
We will also prove that if $\bU$ is as in Proposition~\ref{prop:ideal}, then there is a $1$-parameter subgroup
$\bbG_m$ of $\bT$ whose action is positive, so this includes Proposition~\ref{prop:ideal} as a special case.

\subsubsection{Cases}
Most of this paper will be devoted to Theorem \ref{maintheorem:positiveideal}.  Its proof is quite different
depending on the characteristics of $k$ and $\Field$:
\begin{itemize}
\item[(a)] $\Char(k) = 0$.
\item[(b)] $\Char(k) = p$ is positive and $\Char(\Field) \neq \Char(k)$.
\item[(c)] $\Char(k) = p$ is positive and $\Char(\Field) = \Char(k)$.
\end{itemize}
Case (c) turns out to be quite easy, and does not even require the positive action or for $k$ to be infinite.  
For cases (a) and (b), we need to find appropriate generalizations of Lemma \ref{lemma:polylemma}.  

\subsubsection{Characteristic $0$}
When $\Char(k) = 0$, we use deep work
of Philip Hall on representations of nilpotent groups to give a proof that in some sense is quite similar to
the one we gave for Lemma \ref{lemma:polylemma}, though by necessity the details are more abstract.

\subsubsection{Characteristic $p$}
When $\Char(k) = p$ is positive, new
ideas are needed even for $\bG = \GL_2$ since the matrices $\diag(d,1)$ we used there are not always invertible.
Roughly speaking, we will use 
the positive action to ``compress'' the action of our group onto a small subgroup for which our representation
is understandable.  This subgroup must satisfy a lengthy sequence of hard-to-control polynomial conditions.
Since $k$ is infinite, we will be able to use the Chevalley--Warning theorem to ensure that
no matter what those conditions are, they can always be satisfied: the key point is that, because $k$ is infinite, it contains a copy of $\bbF_p^n$ for all $n$, and $n$ can be chosen large enough to satisfy the
conditions of Chevalley--Warning.

\subsection{Outline}
We prove Proposition \ref{prop:augment} in \S \ref{section:buildingsaugmentation}.  Next, in 
\S \ref{section:uni} we give some background about unipotent groups and prove
that Theorem \ref{maintheorem:positiveideal} implies Proposition~\ref{prop:ideal}. Theorem \ref{maintheorem:positiveideal} is then proved in
\S \ref{section:charpequal}--\S \ref{section:charp}.

\subsection{Conventions}

To avoid cluttering our exposition, unless otherwise specified
all subgroups, morphisms, quotients, etc., we discuss involving an algebraic group $\bG$ defined over a field $k$ are themselves
defined over $k$; for instance, instead of saying that something is a parabolic $k$-subgroup
of $\bG$ we will just say that it is a parabolic subgroup of $\bG$.

\section{Buildings and the augmentation}
\label{section:buildingsaugmentation}

In this section, we prove Proposition~\ref{prop:augment}.  Our proof uses the Borel--Tits
structure theory for connected reductive groups, and all results we quote without
proof or reference can be found in \cite[Chapter V]{BorelBook}.  Let $\bG$ be a connected reductive
group over a field $k$, and let $\bbF$ be another field.  To quickly understand what
we are doing, it might be helpful to assume on a first reading that $\bG = \GL_n$.

\subsection{Borel subgroups, unipotent radicals, and tori}
We start by introducing some key subgroups of $\bG$.  See Example \ref{example:gln} below
for what these are in the special case $\bG = \GL_n$.
\begin{itemize}
\item Let $\bB$ be a minimal parabolic subgroup of $\bG$.  Since $k$ is not
assumed to be algebraically closed, $\bB$ might not be a Borel subgroup, but
in this more general context it serves as a suitable replacement.
\item Let $\bT$ be a maximal split torus contained in $\bB$.
\item Let $Z_{\bG}(\bT)$ be the centralizer of $\bT$.  If $\bG$ were a split group
like $\GL_n$, then $Z_{\bG}(\bT)$ would be $\bT$, but in general it can be larger.
The group $Z_{\bG}(\bT)$ is a Levi factor of $\bB$, and in particular is contained
in $\bB$.
\item Let $N_{\bG}(\bT)$ be the normalizer of $\bT$.
\item Let $W = N_{\bG}(\bT)/Z_{\bG}(\bT)$ be the relative\footnote{The usual (or absolute) Weyl group is what one gets by
working over an algebraic closure $\ok$ and letting $\bT$ be a maximal torus defined over $\ok$.  For
split groups like $\bG = \GL_n$ it is the same as the relative Weyl group.}
Weyl group.  This is a finite
reflection group.
\item Let $\bU$ be the unipotent radical of $\bB$.
\end{itemize}
With this notation, we have $\bB = \bU \rtimes Z_{\bG}(\bT)$.

\begin{remark}
All choices of the pair of subgroups $(\bB,\bT)$ are conjugate in $\bG$.
\end{remark}

\begin{example}
\label{example:gln}
For $\bG = \GL_n$, these subgroups are as follows:
\begin{itemize}
\item The group $\bB$ is the Borel subgroup of upper triangular matrices.
\item The maximal split torus $\bT$ is the group of diagonal matrices.
\item The centralizer $Z_{\bG}(\bT)$ is just $\bT$.
\item The normalizer $N_{\bG}(\bT)$ is the group of monomial matrices, i.e., matrices
with a single nonzero entry in each row and column.
\item The Weyl group $W = N_{\bG}(\bT)/Z_{\bG}(\bT) = N_{\bG}(\bT) / \bT$ is the symmetric group on
$n$ letters, which can be identified with the group of permutation matrices.
\item The unipotent radical $\bU$ is the group of upper triangular matrices with $1$'s on the diagonal.
\end{itemize}
For these, it is clear that $\bB = \bU \rtimes \bT$.
\end{example}

\subsection{Tits building, chambers, and the Steinberg representation}
As described in \cite{TitsBuildings}, the Tits building $\Tits(\bG)$ is the Tits building associated to the
group $\bG(k)$ with the BN-pair $(\bB(k),N_{\bG}(\bT)(k))$.  See \cite[\S 6]{AbramenkoBrownBuildings} for
a textbook reference on the Tits building associated to a BN-pair\footnote{Be warned that the standard
notation in the theory of BN-pairs involves a group $T$, but this is {\em not} $\bT(k)$.  Instead,
it is $Z_{\bG}(\bT)(k)$.  If $k$ is algebraically closed, then $Z_{\bG}(\bT)(k) = \bT(k)$, but in general
it is larger.}.  We will not need to know the complete construction and structure of $\Tits(\bG)$, but
only a few properties of it that we will try to isolate.  

Letting $r$ be the semisimple rank of $\bG$,
the building $\Tits(\bG)$ is an $(r-1)$-dimensional simplicial complex\footnote{If $r=0$, this
means that $\Tits(\bG) = \emptyset$; see Remark \ref{remark:anisotropic} for our conventions
about the empty set.} whose simplices are in bijection
with the proper parabolic subgroups\footnote{For $\bG = \GL_n$, these are the stabilizers
of nontrivial flags $0 \subsetneq V_0 \subsetneq \cdots \subsetneq V_i \subsetneq k^n$.} of $\bG$.  
The conjugation action of $\bG(k)$ on itself permutes these parabolic subgroups, and thus
induces an action of $\bG(k)$ on $\Tits(\bG)$.

Let $\RC_{\bullet}(\Tits(\bG);\bbF)$ be the reduced
simplicial chain complex of $\Tits(\bG)$.  The $(r-1)$-dimensional simplices of $\Tits(\bG)$ are in bijection
with the minimal proper parabolic subgroups, and are called the {\em chambers}.\footnote{For $\bG = \GL_n$, the minimal
proper parabolic subgroups are the stabilizers of maximal flags, or equivalently the conjugates of the Borel subgroup
$\bB$ of upper triangular matrices.}  Let $\minPara$ be
the set of minimal proper parabolic subgroups of $\bG$, so $\RC_{r-1}(\Tits(\bG);\bbF) \cong \bbF[\minPara]$.
Since $\Tits(\bG)$ is an $(r-1)$-dimensional simplicial chain complex, we have $\RC_r(\Tits(\bG);\bbF) = 0$ and
thus
\begin{align*}
\St(\bG;\bbF) &= \RH_{r-1}(\Tits(\bG);\bbF) \\
&= \ker(\RC_{r-1}(\Tits(\bG);\bbF) \stackrel{\partial}{\longrightarrow} \RC_{r-2}(\Tits(\bG);\bbF)) \\
&= \ker(\bbF[\minPara] \stackrel{\partial}{\longrightarrow} \RC_{r-2}(\Tits(\bG);\bbF)).
\end{align*}
In particular, $\St(\bG;\bbF)$ is a subrepresentation of $\bbF[\minPara]$.

\subsection{Apartments}
The homology group $\St(\bG;\bbF)$ is spanned by the {\em apartment classes}.  These are the homology classes
of oriented subcomplexes of $\Tits(\bG)$ that are isomorphic to simplicial triangulations of an $(r-1)$-sphere.  In fact,
these subcomplexes are isomorphic to the Coxeter complex\footnote{For $\bG = \GL_n$, the Weyl group is $S_n$ and the Coxeter
complex is the first barycentric subdivision of an $(n-1)$-simplex, or equivalently the simplicial complex whose $i$-simplices
are $(i+1)$-element subsets of $\{1,\ldots,n\}$.} of the Weyl group $W$.  One example of
an apartment is as follows.  Since $W$ is a finite reflection group, each $w \in W$ has a sign
$(-1)^w$.  Since the group $Z_{\bG}(\bT)$ is contained in $\bB$, for
$w \in W = N_{\bG}(\bT)/Z_{\bG}(\bT)$ the image $w \cdot \bB = w \bB w^{-1}$ makes sense and is an element
of $\minPara$.  We then have an apartment class\footnote{For $\bG = \GL_n$, the chambers $w \cdot \bB$ appearing
in this apartment class are the stabilizers of maximal flags $0 \subsetneq V_0 \subsetneq \cdots \subsetneq V_{n-2} \subsetneq k^n$
such that each $V_i$ is the span of $(i+1)$ standard basis vectors in $k^n$.}
\[\cA_0 = \sum_{w \in W} (-1)^w w \cdot \bB \in \St(\bG;\bbF) \subset \bbF[\minPara].\]
The group $\bG(k)$ acts transitively on the set of apartment classes.

\subsection{Basis for Steinberg}
The apartment classes are not linearly independent.
One version of the Solomon--Tits Theorem (see \cite[Theorem 4.73]{AbramenkoBrownBuildings}) says
that $\St(\bG;\bbF)$ has for a basis the set of apartment classes that ``contain $\bB$'' in the sense
that as an element of $\bbF[\minPara]$ their $\bB$-coefficient is $1$.  Letting $\apart_{\bB}$ be the set
of such apartment classes, we thus have a vector space isomorphism 
$\St(\bG;\bbF) \cong \bbF[\apart_{\bB}]$.  However, since $\bG(k)$ does not preserve
the set $\apart_{\bB}$, it is difficult to understand the $\bG(k)$-action on
$\St(\bG;\bbF)$ using this isomorphism.

\subsection{Group-theoretic interpretation}
Consider the conjugation action of $\bG(k)$ on itself.
A standard property of BN-pairs is that the stabilizer of $\bB(k)$ in $\bG(k)$ under
this conjugation action acts transitively on $\apart_{\bB}$.  Since $\bB$ is a parabolic subgroup,
we have $N_{\bB}(\bG) = \bB$, so the stabilizer of $\bB(k)$ in $\bG(k)$ is $\bB(k)$.  It follows that
$\bB(k)$ acts transitively on $\apart_{\bB}$.  Another standard property of BN-pairs is that the 
stabilizer of the apartment class $\cA_0$ is $Z_{\bG}(\bT)(k)$.  Since $\bB(k)$ is the semi-direct product of 
$\bU(k)$ and $Z_{\bG}(\bT)(k)$, the set map $\alpha\colon \bU(k) \rightarrow \apart_{\bB}$ defined
by $\alpha(g) = g \cdot \cA_0$ is a bijection.

The map $\alpha$ is $\bU(k)$-equivariant with respect to the left action of $\bU(k)$ on itself. 
It is also $Z_{\bG}(\bT)(k)$-equivariant with respect to its conjugation action on $\bU(k)$; indeed, 
for $h \in Z_{\bG}(\bT)(k)$ and $g \in \bU(k)$, we have
\begin{displaymath}
h \cdot \alpha(g)=hg \cdot \cA_0=hgh^{-1} \cdot \cA_0=\alpha(hgh^{-1}),
\end{displaymath}
where in the second step we used that $h$ stabilizes $\cA_0$.  

\subsection{Augmentation}
Let $\iota\colon \St(\bG;\bbF) \rightarrow \bbF[\bU(k)]$ be the composition
\[\St(\bG;\bbF) \stackrel{\cong}{\longrightarrow} \bbF[\apart_{\bB}] \stackrel{\alpha^{-1}}{\longrightarrow} \bbF[\bU(k)],\]
so $\iota$ is a linear isomorphism.  By the above, 
$\iota$ is equivariant for both $\bU(k)$ and $Z_{\bG}(\bT)(k)$; however, in what follows, we will only use the equivariance for $\bU(k)$ and $\bT(k)$.
Let $\epsilon\colon \bbF[\bU(k)] \rightarrow \bbF$ be the augmentation.  The composition $\epsilon \circ \iota\colon \St(\bG;\bbF) \rightarrow \Field$
is $\bB(k)$-invariant but not $\bG(k)$-invariant.  It has the following simple interpretation: 

\begin{lemma}
\label{lemma:describeaug}
Let the notation be as above, and let $x \in \St(\bG;\bbF)$.  Then $\epsilon(\iota(x)) \in \bbF$ is the
$\bB$-coefficient of $x$ considered as an element of $\bbF[\minPara]$.
\end{lemma}
\begin{proof}
It is enough to check this on the basis $\apart_{\bB}$ for $\St(\bG;\bbF)$, so consider $x \in \apart_{\bB}$.  By
the definition of $\apart_{\bB}$, the $\bB$-coefficient of $x$ is $1$, so we must check that $\epsilon(\iota(x))=1$.  By definition, 
$\iota(x)=\alpha^{-1}(x)$. Since this is an element of $\bU(k)$, its image under $\epsilon$ is $1$, as desired.
\end{proof}

\subsection{Proof of Proposition~\ref{prop:augment}}
We finally turn to the proof of Proposition~\ref{prop:augment}. Recall that this states that for all nonzero $x \in \St(\bG;\bbF)$, 
there exists $g \in \bG(k)$ such that $\epsilon(\iota(gx)) \neq 0$.  

\begin{proof}[Proof of Proposition~\ref{prop:augment}]
Consider a nonzero $x \in \St(\bG;\bbF)$.  Regarding $x$ as an element of $\bbF[\minPara]$, some coefficient must
be nonzero.  Since $\bG(k)$ acts transitively on $\minPara$, there exists $g \in \bG(k)$ such
that $gx$ has a nonzero $\bB$-coefficient.  By Lemma \ref{lemma:describeaug}, we
have $\epsilon(\iota(gx)) \neq 0$.
\end{proof}

\section{Unipotent groups and positive actions}
\label{section:uni}

It remains to prove Theorem \ref{maintheorem:positiveideal} and to show that Theorem \ref{maintheorem:positiveideal}
implies Proposition~\ref{prop:ideal}.  The proof of Theorem \ref{maintheorem:positiveideal} starts in 
\S \ref{section:charpequal}.  This section contains preliminaries about unipotent groups and positive
actions.  The final section (\S \ref{section:unipotentradicalpositive}) proves that Theorem \ref{maintheorem:positiveideal} implies
Proposition~\ref{prop:ideal}.  Throughout this section, we fix a field $k$.  All algebraic groups we discuss are affine algebraic group
schemes over $k$.

\subsection{Generalities about unipotent groups}

Recall \cite[IV.2.2.1]{DemazureGabriel} that a {\em unipotent group} is an algebraic group such that every non-trivial closed subgroup admits a non-trivial homomorphism to the additive group $\bbG_a$; equivalently \cite[IV.2.2.5]{DemazureGabriel}, there exists an embedding into the group of strictly upper triangular matrices in $\GL_n$ for some $n$. If $\ok$ is an algebraic closure
of $k$, then an algebraic group $\bU$ is unipotent if and only if its base change $\bU_{\ok}$ is \cite[IV.2.2.6]{DemazureGabriel}.
A unipotent group $\bU$ is {\em split} if there exists a central series
\[\bU = \bU_1 \rhd \bU_2 \rhd \cdots \rhd \bU_n \rhd \bU_{n+1} = 1,\]
where the $\bU_i$ are closed subgroups such that $\bU_i/\bU_{i+1} \cong \bbG_a$ for $1 \leq i \leq n$.  If $k$ is algebraically
closed, then all smooth connected unipotent groups over $k$ are split \cite[IV.4.3.4, IV.4.3.14]{DemazureGabriel}.  From this, we
deduce the following:

\begin{proposition} 
\label{proposition:unipotentnilpotent}
Let $\bU$ be an $n$-dimensional smooth connected unipotent group over a field $k$.  Then $\bU(k)$ is a nilpotent group.  Moreover, if $\Char(k) = p$ is positive,
then all finitely generated subgroups of $\bU(k)$ are finite $p$-groups of nilpotence class at most $n$ and exponent at most $p^n$.
\end{proposition}
\begin{proof}
Let $\ok$ be an algebraic closure of $k$.  Then $\bU_{\ok}$ is split, so $\bU(\ok)$ has a central series 
\[\bU(\ok) = U_1 \rhd U_2 \rhd \cdots \rhd U_n \rhd U_{n+1} = 1\]
with $U_i/U_{i+1} \cong \ok$ for $1 \leq i \leq n$.  This $n$ is
the same as the dimension of $\bU$.  In particular, $\bU(\ok)$ is a nilpotent group.  Since $\bU(k) \subset \bU(\ok)$, 
it follows that $\bU(k)$ is a nilpotent group.

Assume now that $\Char(k) = p$ is positive, and let $G$ be a finitely generated subgroup of $\bU(k)$.
Regard $G$ as a subgroup of $\bU(\ok)$, and let $G_i = G \cap U_i$.  We thus have a central
series
\[G = G_1 \rhd G_2 \rhd \cdots \rhd G_n \rhd G_{n+1} = 1\]
with $G_i/G_{i+1}$ a subgroup of $\ok$ for $1 \leq i \leq n$.  Subgroups of finitely generated
nilpotent groups are finitely generated, so $G_i$ and hence $G_{i}/G_{i+1}$ is finitely generated.
Finitely generated additive subgroups of $\ok$ are isomorphic to finite products of $\Z/p\Z$.  We conclude
that $G$ is a finite $p$-group of nilpotence class at most $n$ and exponent at most $p^n$.
\end{proof}

\subsection{Splitting off unipotent subgroups}
Unipotent groups $\bU$ can be studied inductively by identifying normal subgroups $\bU'$
and then studying the unipotent groups $\bU'$ and $\bU/\bU'$.  The following is helpful for combining results about $\bU'$ and $\bU/\bU'$
into results about $\bU$:

\begin{proposition}
\label{proposition:unipotentproduct}
Let $\bG$ be a linear algebraic group and let $\bV \lhd \bG$ be a smooth connected split unipotent normal subgroup.
There exists a subvariety $\bX$ of $\bG$ containing the identity such that the map  $\bV \times \bX \rightarrow \bG$ induced
by the product on $\bG$ is an isomorphism of varieties.
\end{proposition}
\begin{proof}
Let $\bX = \bG/\bV$ and let $\pi\colon \bG \rightarrow \bX$ be the quotient map.  The map $\pi$ gives $\bG$ the
structure of a $\bV$-torsor over $\bX$, and we can embed $\bX$ into $\bG$ as in the proposition precisely when
that torsor is trivial.  The result thus follows from two facts: the algebraic group $\bX$ is affine (see \cite[III.3.5.6]{DemazureGabriel}),
and all such torsors over affine bases are trivial (see \cite[IV.4.3.7]{DemazureGabriel}).
\end{proof}

\subsection{Weights and positive actions}
\label{section:weightspositive}

Recall (\cite[\S III.8.17]{BorelBook} or \cite[IV.1.1.6]{DemazureGabriel}) that if $\bbG_m = \GL_1$ acts algebraically and linearly on a $k$-vector space $V$, then $V$ decomposes as a direct
sum of weight spaces $V_d$, where $\bbG_m(k) = k^{\times}$ acts on $V_d$ as $t \cdot v = t^d v$ for
$t \in k^{\times}$ and $v \in V_d$.  The integers $d$ with $V_d \neq 0$ are called the {\em weights} of the action.  We say that $\bbG_m$ acts on $V$ with {\em positive weights} if each weight is positive. 

For instance, consider an action of $\bbG_m$ on $\bbG_a$. For $t \in \bbG_m(k)=k^{\times}$ and $x \in \bbG_a(k)=k$, write $\presup{t}{x}$ for the action of $t$ on $x$. The action of $\bbG_m$ on $\bbG_a$ is automatically linear. Indeed, the $k$-algebra automorphisms of $k[x]$ have the form $x \mapsto a+bx$ with $a,b \in k$, and so any automorphism of $\bbG_a$ as a group scheme has the form $x \mapsto bx$. We thus see that the action has a single weight $m \in \Z$, and so $\presup{t}{x} = t^m x$ for all $t$ and $x$ as above.

\begin{remark}
In finite characteristic, there exist non-linear actions of $\bbG_m$ on $\left(\bbG_a\right)^{\times 2}$.
See Remark \ref{remark:nonlinear} below for an example.
\end{remark}

Next, let $\bU$ be a smooth connected unipotent group over $k$ equipped with an action of $\bbG_m$.  For $t \in \bbG_m(k)$ and $g \in \bU(k)$, we will
denote the action of $t$ on $g$ by $\presup{t}{g}$.  As in the introduction, we say that the action of $\bbG_m$ on $\bU$ is a {\em positive action}
if the induced action on the Lie algebra $\Lie(\bU)$ has positive weights.  We then have the following.

\begin{proposition}
\label{proposition:centerpositive}
Let $\bU$ be a smooth connected unipotent group over a field $k$ equipped with a positive $\bbG_m$-action.  Then
there exists a $\bbG_m$-stable central subgroup $\bA \lhd \bU$ with $\bA \cong \bbG_a$ such that $\bbG_m$ acts on $\bA$ with
positive weight.
\end{proposition}
\begin{proof}
First suppose that $\Char(k)=0$.  In this case, as discussed in \cite[III.6.3.7]{DemazureGabriel} and \cite[IV.2.4]{DemazureGabriel} we have an isomorphism
$\Lie(\bU) \rightarrow \bU$ given by the exponential map (which can be expressed by a {\em finite} power series since our group is nilpotent).  The
construction of $\Lie(\bU) \rightarrow \bU$ is functorial, and in particular is $\bbG_m$-equivariant.  We can then take
$\bA$ to be the image of a $1$-dimensional  weight space of $\Lie(\bU)$ contained in its center.

Now suppose that $\Char(k)=p$ is positive. Let $\bU_1$ be {\em cckp-kernel} of $\bU$, i.e., the maximal smooth connected $p$-torsion central 
closed subgroup of $\bU$; this exists and is non-trivial \cite[\S B.3]{ConradGabberPrasad}.  The subgroup $\bU_1$ is stable
under automorphisms of $\bU$, and is therefore $\bbG_m$-stable.  Tits \cite[Theorem~B.4.3]{ConradGabberPrasad} proved that
$\bU_1=\bU_2 \times \bU_3$ where $\bU_2$ and $\bU_3$ are closed smooth $\bbG_m$-stable subgroups of $\bU_1$ with the following
properties:
\begin{itemize}
\item The group $\bbG_m$ acts trivially on $\bU_2$.
\item There is vector group $\bV$ defined over $k$ (i.e., a vector space over $k$ regarded as an algebraic
group via its additive structure) and a linear action of $\bbG_m$ on $\bV$ such that
$\bU_3$ is $\bbG_m$-equivariantly isomorphic to $\bV$.
\end{itemize}
Since $\Lie(\bU)$ only has positive weights, we have $\Lie(\bU_2)=0$, so $\bU_2 = 1$.
Thus $\bU_1=\bU_3$ is $\bbG_m$-equivariantly isomorphic to $\bV$.  
We can now take $\bA$ to be a subgroup of $\bU_1$ 
corresponding to a weight space of $\bV$ under this isomorphism.
\end{proof}

\begin{remark}
\label{remark:nonlinear}
If $\Char(k)=p$ then there are non-linear actions of $\bbG_m$ on vector spaces over $k$. For example, let $V=\left(\bbG_a\right)^{\times 2}$, let $\sigma$ be the linear action of $\bbG_m$ on $V$ given by $\sigma(t)(x,y)=(tx,t^2y)$, and let $\tau$ be the automorphism of $V$ given by $\tau(x,y)=(x,y+x^p)$. Then conjugating $\sigma$ by $\tau$ gives a non-linear action of $\bbG_m$ on $V$. This demonstrates one of the difficulties that Tits' theorem must handle.
\end{remark}

\begin{remark}
Applying Proposition \ref{proposition:centerpositive} repeatedly, one can show that a smooth connected unipotent group $\bU$ over a field $k$ equipped with
a positive $\bbG_m$-action must be split.  This implies in particular that as a variety, $\bU$ is isomorphic to an affine space over $k$.
\end{remark}

\subsection{Characteristic 0}
The following proposition will be the key to understanding positive actions in characteristic $0$:

\begin{proposition}
\label{proposition:splitpositivechar0}
Let $\bU$ be an $n$-dimensional smooth connected unipotent group over a field $k$ of characteristic $0$ equipped with a positive $\bbG_m$-action.
Then there exist $\bbG_m$-stable subgroups $\bG_1,\ldots,\bG_n$ of $\bU$ with the following properties:
\begin{itemize}
\item The map $\bG_1 \times \cdots \times \bG_n \rightarrow \bU$ arising from the product on $\bU$ is an isomorphism
of varieties.
\item For $1 \leq i \leq n$, we have $\bG_i \cong \bbG_a$ and $\bbG_m$ acts on $\bG_i$ with positive weight.
\end{itemize}
\end{proposition}
\begin{proof}
Let $\Lie(\bU) = \oplus_{i=1}^n \fu_i$ be a decomposition into $1$-dimensional weight spaces and let $\bG_i$
be the image of $\fu_i$ under the exponential map (see \cite[\S IV.2.4.5]{DemazureGabriel}).
\end{proof}

\subsection{Extending positive actions}
Let $\overline{\bbG}_m = \Spec(k[x])$, which is an algebraic monoid under multiplication; this space is obtained from $\bbG_m$ by simply adding the point $0$.  The following proposition
shows that if a $\bbG_m$-action on a unipotent group is positive, then it can be extended to an action of $\overline{\bbG}_m$.

\begin{proposition}
\label{proposition:extendpositive}
Let $\bU$ be a smooth connected unipotent group over a field $k$ equipped with a positive $\bbG_m$-action.
Then the $\bbG_m$-action can be extended to an action of $\overline{\bbG}_m$ such that $\presup{0}{g} = \id$
for all $g \in \bU(k)$.
\end{proposition}
\begin{proof} 
It is enough to prove this for the base change to an algebraic closure of $k$, so without loss of generality we
can assume that $k$ is algebraically closed.  The proof will be by induction on $\dim(\bU)$.  The base
case is $\dim(\bU) = 1$.  Since $k$ is algebraically closed,
$\bU$ is split, so we must have $\bU \cong \bbG_a$.  The single positive weight $d \geq 1$ of
the action of $\bbG_m$ on $\Lie(\bU) \cong \bbG_a$ satisfies the key identity
\[\presup{t}{x} = t^d x \quad \text{for $t \in \bbG_m(k) = k^{\times}$ and $x \in \bbG_a(k) = k$}.\]
Since $d$ is not negative, this also makes sense for $t=0$, so this extends to $\overline{\bbG}_m$.  Finally,
since $d \neq 0$ this extension satisfies $\presup{0}{x} = 0$.

Assume now that $\dim(\bU) \geq 2$ and that the proposition is true for smaller dimensions.
We will first prove that the $\bbG_m$ action extends to $\overline{\bbG}_m$.
Let $R$ be the ring of regular functions $f\colon \bU(k) \rightarrow k$.  Since
$\bbG_m$ acts algebraically on $\bU$, it also acts algebraically and linearly on
$R$.  Proving that the action of $\bbG_m$ on $\bU$ extends to an action
of $\overline{\bbG}_m$ is equivalent to proving that the action of $\bbG_m$
on $R$ extends to an action of $\overline{\bbG}_m$.

We say that a regular function $f\colon \bU(k) \rightarrow k$ is a {\em weight function} of weight $n$ if 
\begin{equation}
\label{eqn:regularweight}
f(\presup{t}{g}) = t^n f(g) \quad \text{for $t \in \bbG_m(k)$ and $g \in \bU(k)$}.
\end{equation}
Since $\bbG_m$ acts algebraically and linearly on $R$, as a vector space $R$
decomposes into a direct sum of $1$-dimensional subspaces spanned by weight functions (see the beginning\footnote{One might worry that there could
be an issue since $R$ is an infinite-dimensional vector space, but the references given
in \S \ref{section:weightspositive} apply in this level of generality.  In fact, it
follows from algebraicity that $R$ is a union of finite-dimensional subrepresentations.} of \S \ref{section:weightspositive}). 
The $\bbG_m$-action on $R$ extends to $\overline{\bbG}_m$ if and only if there are no nonzero weight functions of negative weight: the point is that just like in the case where $\dim(\bU) = 1$, this implies
that \eqref{eqn:regularweight} also makes sense for $t=0$.

By Proposition \ref{proposition:centerpositive}, there exists a $\bbG_m$-stable central subgroup $\bA \lhd \bU$ with $\bA \cong \bbG_a$
such that $\bbG_m$ acts on $\bA$ with positive weight.  Let $\bU'=\bU/\bA$, so we have an exact sequence
\[1 \longrightarrow \bA \longrightarrow \bU \stackrel{\pi}{\longrightarrow} \bU' \longrightarrow 1.\]
There is a corresponding short exact sequence of Lie algebras, from which it follows that the induced $\bbG_m$-action on $\bU'$ is positive.
By induction on the dimension, the $\bbG_m$-action on $\bU'$ extends to an action of $\overline{\bbG}_m$ such that $\presup{0}{g} = \id$
for all $g \in \bU'(k)$.

Since the identity element of $\bU'(k)$ is fixed by $\bbG_m$, the ideal of regular functions vanishing on it can be generated by 
weight functions $\overline{h}_1, \ldots, \overline{h}_r$, necessarily of nonnegative weights. 
Let $h_i=\pi^{\ast}(\overline{h}_i)$.  The $h_i$ generate the ideal $I \subset R$ of regular functions on $\bU(k)$ vanishing on $\bA(k)$. 

Now, suppose $f$ is a weight function on $\bU(k)$ of negative weight. Then $f|_{\bA(k)}$ is a weight function on $\bA(k)$ of negative weight, 
so $f|_{\bA(k)} = 0$.  It follows that $f \in I$, so we can write $f=\sum_{i=1}^r g_i h_i$ for regular functions $g_i$; in fact, we can take 
the $g_i$ to be weight functions such that $g_i h_i$ has the same weight as $f$. Since $f$ has negative weight and $h_i$ 
has non-negative weight, it follows that $g_i$ must have negative weight. 
Hence, by the same argument, $g_i \in I$, so $f \in I^2$. Continuing in this manner, we find that
\[f \in \bigcap_{n \ge 1} I^n.\]
Since $R$ is a domain, it follows from the Krull intersection theorem that $\bigcap_{n \geq 1} I^n = 0$ (see \cite[Corollary 5.4]{Eisenbud}), so
$f=0$.  We thus find that the $\bbG_m$-action on $\bU$ extends to $\overline{\bbG}_m$. 

It remains to show that for $g \in \bU(k)$, we have $\presup{0}{g} = \id$.  We know this
for $\bU'$ by induction, so
\[\pi(\presup{0}{g})=\presup{0}{\pi(g)} = \id.\]
It follows that $\presup{0}{g} \in \bA(k)$.  Since $\bbG_m$ acts on $\bA$ with positive weight, a final application of our inductive hypothesis
says that $\presup{0}{h}=\id$ for any $h \in \bA(k)$. Thus $\presup{0}{g} = {}^0{(\presup{0}{g})} = \id$,
as desired.
\end{proof}

\begin{remark}
The converse also holds: if $\bU$ is a smooth connected unipotent group over a field $k$ equipped with an
action of $\bbG_m$ that extends to $\overline{\bbG}_m$ such that $\presup{0}{g} = \id$
for all $g \in \bU(k)$, then the $\bbG_m$-action is positive.
\end{remark}

\subsection{Positive actions on unipotent radicals}
\label{section:unipotentradicalpositive}

Our final result in this section shows that Theorem \ref{maintheorem:positiveideal} implies Proposition~\ref{prop:ideal}.

\begin{proposition}
\label{prop:radicalpositive}
Let $\bG$ be a connected reductive group over a field $k$, let $\bB$ be a minimal parabolic subgroup
of $\bG$, let $\bU$ be the unipotent radical of $\bB$, and let $\bT$ be a maximal split torus
of $\bB$.  Then there exists a one-parameter subgroup $\OneParam\colon \bbG_m \rightarrow \bT$ that acts positively on $\bU$.
\end{proposition}

\begin{example}
Suppose $\bG=\GL_n$, the group $\bB$ is the Borel subgroup of upper triangular matrices, $\bT$ is the
torus of diagonal matrices, and $\bU$ is the unipotent subgroup of upper triangular matrices
with $1$'s on the diagonal.  We can then take $\OneParam\colon \bbG_m \rightarrow \bT$
to be the $1$-parameter subgroup $\OneParam(t) = \diag(t^n,t^{n-1},\ldots,t^1)$.
The key property of $\OneParam$ is that for $g \in \bU(k)$ and $t \in \bbG_m(k)$, the matrix
$\presup{t}{g}=\OneParam(t) g \OneParam(t)^{-1}$ is obtained from $g$ by multiplying every
entry above the diagonal by a {\em positive} power of $t$.
\end{example}

\begin{proof}[Proof of Proposition \ref{prop:radicalpositive}]
We start by recalling some basic facts about relative root systems (see \cite[\S 21]{BorelBook}).
Let $X(\bT)$ be the group of characters $\chi\colon \bT \rightarrow \bbG_m$ and let $Y(\bT)$ be the group
of one-parameter subgroups $\gamma\colon \bbG_m \rightarrow \bT$.  For $\chi \in X(\bT)$ and
$\gamma \in Y(\bT)$, the composition $\chi \circ \gamma \colon \bbG_m \rightarrow \bbG_m$ can be written
in the form $\chi \circ \gamma(t) = t^n$ for some $n \in \Z$.  Define $\langle \chi, \gamma \rangle = n$.  This
extends to a nondegenerate pairing between $X(\bT) \otimes \Q$ and $Y(\bT) \otimes \Q$.
Let $\RelRoot \subset X(\bT)$ be the relative root system of $\bT$ in $\bG$.  There is a choice of positive
roots $\RelRootPlus \subset \RelRoot$ such that $\RelRootPlus$ is precisely the set of weights of the action
of $\bT$ on the Lie algebra of $\bU$.  Let $\{\chi_1,\ldots,\chi_n\} \subset \RelRootPlus$ be the simple
positive roots.  These form a basis for $X(\bT) \otimes \Q$, and every element of $\RelRootPlus$
is a nonnegative integer linear combination of the $\chi_i$. 
Using the nondegeneracy of our pairing, we can find some $\alpha \in Y(\bT) \otimes \Q$
such that $\langle \chi_i,\alpha \rangle = 1$ for $1 \leq i \leq n$.  Choosing $d \geq 1$ such
that $d \alpha \in Y(\bT)$, we can then take $\OneParam = d \alpha$.
\end{proof}
 
\section{Ideals and positive actions I: equal characteristic \texorpdfstring{$p$}{p}}
\label{section:charpequal}

We now turn to the proof of Theorem \ref{maintheorem:positiveideal}.  The following result implies this theorem in the special case where $\Char(k) = \Char(\Field) = p$ is positive, but is more
general since it does not require $\bU$ to be equipped with a positive action or for $k$ to be infinite:

\begin{proposition}
\label{prop:unipotentgroupringeasy}
Let $\bU$ be a smooth connected unipotent group over a field $k$ and let $\Field$ be another field.  Assume
that $\Char(k) = \Char(\Field) = p$ is positive.  Let $I$ be a left ideal in $\bbF[\bU(k)]$ that does not lie in the augmentation
ideal.  Then $I = \bbF[\bU(k)]$.
\end{proposition}
\begin{proof}
Let $x \in I$ be an element that does does not belong to the augmentation ideal.  Write
$x = \sum_{i=1}^n c_i [g_i]$ with $c_i \in \bbF$ and $g_i \in \bU(k)$.  Let $\Gamma$ be the subgroup
of $\bU(k)$ generated by the $g_i$.  Proposition \ref{proposition:unipotentnilpotent} implies
that $\Gamma$ is a finite $p$-group.  Let $\epsilon_{\Gamma}\colon \bbF[\Gamma] \rightarrow \bbF$ be the augmentation.
Then $\ker(\epsilon_{\Gamma})$ is the Jacobson radical of $\bbF[\Gamma]$, i.e., the intersection of all maximal
left ideals (see \cite[Corollary 8.8]{LamRings}).  Since $\ker(\epsilon_{\Gamma})$ is itself a maximal
left ideal, it follows that it is the unique maximal left ideal.  The element $x \in I \cap \bbF[\Gamma]$
does not lie in $\ker(\epsilon_{\Gamma})$, so $I \cap \bbF[\Gamma] = \bbF[\Gamma]$.  In particular,
$[\id] \in I$, so $I = \bbF[\bU(k)]$.
\end{proof}

\section{Ideals and positive actions II: characteristic \texorpdfstring{$0$}{0}}
\label{section:char0}

In this section, we prove Theorem~\ref{maintheorem:positiveideal} when $\Char(k)=0$.  The proof
is in \S \ref{section:assemblechar0} after three sections of preliminaries.

\subsection{Modules over nilpotent groups}
\label{section:hall}

Recall that a group $G$ is \emph{abelian-by-nilpotent} if there is a normal abelian subgroup $A$ of $G$ such that $G/A$ is nilpotent.  
Also, $G$ is \emph{residually finite} if it injects into its profinite completion, or equivalently if the intersection
of all finite-index normal subgroups of $G$ is trivial.  Hall \cite{HallResiduallyFinite} proved the following.  See
\cite[Theorem 4.3.1]{LennoxRobinson} for a textbook reference.

\begin{theorem}[Hall]
\label{theorem:hall}
All finitely generated abelian-by-nilpotent groups are residually finite.
\end{theorem}

We will apply Theorem \ref{theorem:hall} in the form of the following corollary.  If $M$ is a module
over a ring $R$, then an $R$-submodule $M'$ of $M$ is said to be {\em finite-index} if the cardinality
of $M/M'$ is finite.  Just like for groups, we say that $M$ is {\em residually finite} if the intersection
of all finite-index submodules is $0$.

\begin{corollary}
\label{corollary:separatenil}
Let $G$ be a finitely generated nilpotent group and let $M$ be a finitely generated $\Z[G]$-module.  Then $M$ is residually
finite.
\end{corollary}
\begin{proof}
Set $\Gamma = M \rtimes G$, so $\Gamma$ is an abelian-by-nilpotent group.  Since
$G$ is finitely generated and $M$ is a finitely generated $\Z[G]$-module, the group $\Gamma$ is finitely generated.
Theorem~\ref{theorem:hall} thus implies that $\Gamma$ is residually finite, so the intersection
of all elements of $\cF = \Set{$\Delta$}{$\Delta \lhd \Gamma$ finite index}$ is trivial.
This implies that the intersection of all elements
of $\cF' = \Set{$\Delta \cap M$}{$\Delta \in \cF$}$ is $0$.  Each element of $\cF'$ is a finite-index submodule of $M$, so we
conclude that $M$ is residually finite.
\end{proof} 

\subsection{Finitely generated subgroups of unipotent groups}
\label{section:constructchar0}

For a group $\Gamma$ and $d \geq 1$, let $\Gamma(d)$ be the subgroup of $\Gamma$ generated by all $d^{\text{th}}$ powers. 

\begin{proposition}
\label{prop:constructgammachar0}
Let $\bU$ be a smooth connected unipotent group over a field $k$ of characteristic $0$ equipped with a positive $\bbG_m$-action.
Let $S$ be a finite subset of $\bU(k)$.  There exists a finitely generated subgroup $\Gamma$ of $\bU(k)$ with $S \subset \Gamma$ such that
$\presup{d}{g} \in \Gamma(d)$ for all $g \in \Gamma$ and $d \geq 1$.
\end{proposition}
\begin{proof}
By Proposition \ref{proposition:splitpositivechar0}, there exist $\bbG_m$-stable subgroups $\bG_1,\ldots,\bG_n$ of $\bU$ 
such that the following hold:
\begin{itemize}
\item The map $\bG_1 \times \cdots \times \bG_n \rightarrow \bU$ arising from the product on $\bU$ is an isomorphism
of varieties.
\item For $1 \leq i \leq n$, we have $\bG_i \cong \bbG_a$ and $\bbG_m$ acts on $\bG_i$ with positive weight $m_i \geq 1$.
\end{itemize}
This second condition implies in particular that $\presup{d}{x} = x^{d^{m_i}}$ for all $x \in \bG_i(k)$ and $d \geq 1$.
For $1 \leq i \leq n$, choose a finite set $S_i \subset \bG_i(k)$ such that every element of $S$ is a product of elements
of the $S_i$.  Let $\Gamma$ be the subgroup of $\bU(k)$ generated by $S_1 \cup \cdots \cup S_n$.  We thus have $S \subset \Gamma$. 
For $g \in \Gamma$, we can write
\[g = s_1 \cdots s_r \quad \text{with $s_j \in S_{i_j}$ for $1 \leq j \leq r$},\]
so for $d \geq 1$ we have
\[\presup{d}{g} = \presup{d}{s_1} \cdots \presup{d}{s_r} = (s_1)^{d^{m_{i_1}}} \cdots (s_r)^{d^{m_{i_r}}} \in \Gamma(d).\qedhere\]
\end{proof}

\subsection{Modules over unipotent groups in char 0}

The following is our generalization of Lemma \ref{lemma:polylemma} to the setting of unipotent groups over fields of characteristic $0$.

\begin{proposition} 
\label{proposition:polylemmachar0}
Let $\bU$ be a smooth connected unipotent group over a field $k$ of characteristic~0 equipped with a positive $\bbG_m$-action. Let $M$ be a $\bbZ[\bU(k)]$-module and let $N \subset M$ be a submodule.  For some $m_1,\ldots,m_n \in M$ and $g_1,\ldots,g_n \in \bU(k)$, assume that
\[\text{$\presup{d}{g_1} \cdot m_1 + \cdots + \presup{d}{g_n} \cdot m_n \in N$ for all $d \geq 1$}.\]
Then $m_1+\cdots+m_n \in N$.
\end{proposition}
\begin{proof}
Replacing $M$ by $M/N$, we can assume that $N=0$. Proposition~\ref{prop:constructgammachar0} says there is a finitely generated
subgroup $\Gamma$ of $\bU(k)$ containing $\{g_1, \ldots, g_n\}$ such that $\presup{d}{g} \in \Gamma(d)$ for all $g \in \Gamma$ and $d \geq 1$.
Proposition~\ref{proposition:unipotentnilpotent} implies that $\Gamma$ is a nilpotent group. Let $M'$ be the $\Z[\Gamma]$-submodule of $M$ 
generated by the $m_i$.  Consider a finite-index submodule $M''$ of $M'$.
The group of automorphisms of the finite abelian group underlying $M'/M''$ is finite.  Letting $d$ be its exponent, the group 
$\Gamma(d)$ acts trivially on $M'/M''$.  Thus
\begin{displaymath}
0= \presup{d}{g_1} \cdot m_1 + \cdots + \presup{d}{g_n} \cdot m_n \equiv m_1+\cdots+m_n \pmod{M''},
\end{displaymath}
so $m_1+\cdots+m_n \in M''$.  Since $M''$ was an arbitrary finite-index submodule of $M'$, Corollary~\ref{corollary:separatenil}
implies that $m_1+\cdots+m_n = 0$, as required.
\end{proof}

\subsection{Conclusion}
\label{section:assemblechar0}

The following is Theorem \ref{maintheorem:positiveideal} in the special case where $\Char(k)=0$.

\begin{theorem} 
\label{theorem:positiveidealchar0}
Let $\bU$ be a smooth connected unipotent group over a field $k$ of characteristic~0 equipped with a positive action of $\bbG_m$ and let
$\bbF$ be another field.  Let $I \subset \bbF[\bU(k)]$ be a left ideal that is stable under $\bbG_m$ and not contained in the augmentation
ideal.  Then $I = \bbF[\bU(k)]$.
\end{theorem}
\begin{proof}
For $g \in \bU(k)$, write $[g]$ for the associated element of $\bbF[\bU(k)]$.  Let $\epsilon\colon \bbF[\bU(k)] \rightarrow \bbF$ be
the augmentation.  Since $I$ is not contained in the augmentation ideal, there
exists some $x \in I$ with $\epsilon(x) = 1$.  Write this as
\[x = \sum_{i=1}^n c_i [g_i] \in I \quad \text{with $g_1,\ldots,g_n \in \bU(k)$, $c_1,\ldots,c_n \in \bbF$, and $\sum_{i=1}^n c_i = 1$}.\]
Since $I$ is stable under the action of $\bbG_m(k)$, for all $t \in \bbG_m(k)$ we have
$\presup{t}{x} \in I$, so
\[\sum_{i=1}^n c_i [\presup{t}{g_i}] = \sum_{i=1}^n \presup{t}{g_i} \cdot c_i [1] \in I.\]
Applying Proposition \ref{proposition:polylemmachar0} with $M=\bbF[\bU(k)]$ and $N=I$, we deduce that
$[1] = \sum_i c_i [1] \in I$, so $I = \bbF[\bU(k)]$.
\end{proof}

\section{Ideals and positive actions III: unequal characteristic \texorpdfstring{$p$}{p}}
\label{section:charp}

In this section, we prove Theorem~\ref{maintheorem:positiveideal} when $\Char(k)=p$ and $\Char(\bbF)\ne p$. The proof is in \S \ref{section:assemblecharp} after five sections of preliminaries.

\subsection{Finitely generated subgroups}
If $\bU$ is a smooth connected unipotent group over a field $k$ of positive characteristic, then 
Proposition \ref{proposition:unipotentnilpotent} says that all finitely generated subgroups of $\bU(k)$ are finite.
The following lemma says that if one bounds the size of a generating set, then only finitely many
isomorphism classes of finite groups occur:

\begin{proposition}
\label{prop:finitelymanysubgroups}
Let $\bU$ be a smooth connected unipotent group over a field $k$ of positive characteristic and
let $m \geq 1$.  Then there only exist finitely many isomorphism classes of subgroups of $\bU(k)$ that are
generated by $m$ elements.
\end{proposition}
\begin{proof}
Let $n = \dim(\bU)$ and $p = \Char(k)$.  Proposition \ref{proposition:unipotentnilpotent} says that all finitely generated subgroups of $\bU(k)$ are
nilpotent of class at most $n$ and have exponent at most $p^n$.  A cheap way to proceed is to
quote the Restricted Burnside Problem (proved by Zelmanov \cite{Zelmanov}), which says that there are only
finitely many isomorphism classes of finite groups with $m$ generators and exponent at most $p^n$.  An easier
approach\footnote{This is actually the first step in the restricted Burnside problem: much of the hard work
in its proof is bounding the nilpotence class of finite $p$-groups in terms of their exponent and number
of generators.} is as follows.  Let $\cC_d$ be set of isomorphism classes of finite groups of nilpotence
class $d$ that have $m$ generators and exponent at most $p^n$.  We will prove that $|\cC_d|<\infty$ by induction
on $d$.

The base case $d=1$ is trivial, so assume that $d>1$.  Consider $G \in \cC_d$.  Let $\gamma(G)$ be the 
$d^{\text{th}}$ term of the lower central series of $G$, so $G/\gamma(G) \in \cC_{d-1}$.  Since
$|\cC_{d-1}|<\infty$, it is enough to prove that there are finitely many possibilities for $\gamma(G)$.  Taking
$d$-fold iterated commutators, we get a surjective map of abelian groups $\wedge^d G^{\text{ab}} \rightarrow \gamma(G)$.
Since $G$ has $m$ generators and exponent at most $p^n$, there are finitely many possibilities for $G^{\text{ab}}$, and
thus also finitely many possibilities for $\gamma(G)$.
\end{proof}

This has the following corollary.  Let $F_m$ denote the free group on generators $\{x_1,\ldots,x_m\}$.
If $G$ is a group, $g_1,\ldots,g_m \in G$ are elements, and $w \in F_m$, then let
$w(g_1,\ldots,g_m) \in G$ be the image of $w$ under the homomorphism $F_m \rightarrow G$ taking $x_i$
to $g_i$ for $1 \leq i \leq m$.

\begin{corollary}
\label{corollary:words}
Let $\bU$ be a smooth connected unipotent group over a field $k$ of positive characteristic.
For all $m \geq 1$, there exists a finite set $\cW_m$ of elements of $F_m$ such that for all
$s_1,\ldots,s_m \in \bU(k)$, the subgroup of $\bU(k)$ generated by the $s_i$ equals
$\Set{$w(s_1,\ldots,s_m)$}{$w \in \cW_m$}$.
\end{corollary}
\begin{proof}
By Proposition \ref{prop:finitelymanysubgroups}, there are only finitely many possibilities for
the isomorphism class of the subgroup generated by the $s_i$.  For each of these groups and
each choice of $m$-element generating for it, include words in $\cW_m$ to express every element
in terms of those generators.
\end{proof}

\subsection{A-polynomials}

Let $k$ be a field of positive characteristic $p$.  A map $\lambda \colon k \rightarrow \Field_p$ is an {\em additive map} if it
is a homomorphism of additive groups.
If $X$ is a variety over $k$, then an {\em A-polynomial} on $X$ is a function $f \colon X(k) \to \Field_p$ of the form $\lambda \circ \phi$, where $\phi \colon X \to \bA^1$ is a morphism of varieties over $k$ and $\lambda \colon k \to \Field_p$ is an additive map.  Here the ``A'' stands for
``additive''.  We will need to do a sort of algebraic geometry with A-polynomials, and the
key result is as follows:

\begin{proposition}
\label{prop:apolynomial}
Let $k$ be an infinite field of positive characteristic $p$, let
$f_1,\ldots,f_r\colon \bbA^1(k) \rightarrow \Field_p$ be A-polynomials such that
$f_i(0) = 0$ for $1 \leq i \leq r$, and let $\fa$ be an infinite additive subgroup
of $k$.  Then there exists some nonzero $a \in \fa$  such that $f_i(a) = 0$ for all $1 \leq i \leq r$.
\end{proposition}
\begin{proof}
Write $f_i = \lambda_i \circ \phi_i$ with $\phi_i \in k[z]$ and $\lambda_i\colon k \rightarrow \Field_p$
an additive map.  Let $d_i = \deg(\phi_i)$, and set $m = 1+ \sum_i d_i$. Regarding
$\fa$ as an infinite-dimensional vector space over $\Field_p$, we can choose $\Field_p$-linearly independent
elements $v_1,\ldots,v_m \in \fa$.  For $1 \leq i \leq r$, define a function
\[h_i\colon \Field_p^m \rightarrow \Field_p, \qquad h_i(x_1,\ldots,x_m) = f_i(x_1 v_1 + \cdots + x_m v_m).\]
We claim that $h_i$ is a polynomial of degree at most $d_i$.  Indeed, $\phi_i$ is a sum of terms of the form
$c z^e$ with $c \in k$ and $e \leq d_i$, so $h_i$ is a sum of terms
of the form $\lambda_i(c (x_1 v_1 + \cdots + x_m v_m)^e)$.  This can be expanded as
\begin{align*}
&\sum_{j_1 + \cdots + j_m = e} \lambda_i\left(\binom{e}{j_1,\ldots,j_m} c x_1^{j_1} v_1^{j_1} \cdots x_m^{j_m} v_m^{j_m}\right) \\
&\quad \quad \quad = \sum_{j_1 + \cdots + j_m = e} x_1^{j_1} \cdots x_m^{j_m} \lambda_i\left(\binom{e}{j_1,\ldots,j_m} cv_1^{j_1} \cdots v_m^{j_m}\right),
\end{align*}
where the $\binom{e}{j_1,\ldots,j_m}$ are multinomial coefficients.   
Here we are using the fact that $\lambda_i\colon k \rightarrow \Field_p$ is an additive map, which implies that it is
$\Field_p$-linear.  Since $f_i(0) = 0$, we also have $h_i(0,\ldots,0) = 0$.  The 
Chevalley--Warning theorem \cite[Corollary I.2.2.1]{SerreArithmetic} thus implies that there exists some
nonzero $(x_1,\ldots,x_m) \in \Field_p^m$ such that $h_i(x_1,\ldots,x_m) = 0$ for all $1 \leq i \leq r$.  The
desired $a \in \fa$ is then $a = x_1 v_1 + \cdots + x_m v_m$.
\end{proof}

\subsection{Subgroups satisfying A-polynomials}
Let $\bU$ be a smooth connected unipotent group over an infinite field $k$ of positive characteristic 
equipped with a positive action of $\bbG_m$.  By Proposition~\ref{proposition:extendpositive}, 
the $\bbG_m$ action on $\bU$ extends to an action of $\overline{\bbG}_m$ satisfying $\presup{0}{g}=\id$ 
for all $g \in \bU(k)$.  For a subset $S$ of $\bU(k)$ and an additive subgroup $\fa$ of $k$, define $\bU(S,\fa)$ 
to be the subgroup of $\bU(k)$ generated by $\Set{$\presup{a}{s}$}{$s \in S$, $a \in \fa$}$. The 
following demonstrates the flexibility of these subgroups:

\begin{proposition}
\label{prop:groupapolynomial}
Let $\bU$ be a smooth connected unipotent group over an infinite field $k$ of positive characteristic $p$
equipped with a positive action of $\bbG_m$.  Let $S$ be a finite subset of $\bU(k)$, let $\fa$ be an infinite additive subgroup of $k$, and let
$f\colon \bU(k) \rightarrow \Field_p$ be an A-polynomial such that $f(\id) = 0$.  
Then there exists an infinite additive subgroup $\fb$ of $\fa$ such that $f$ vanishes on $\bU(S,\fb)$.
\end{proposition}
\begin{proof}
Say that an additive subgroup $\fc$ of $k$ is {\em $f$-vanishing} if $f$ vanishes on $\bU(S,\fc)$.
Below we will construct a strictly increasing chain
$\fc_1 \subsetneq \fc_2 \subsetneq \cdots$
of $f$-vanishing finite additive subgroups of $\fa$.  Having done this,
the union $\fb$ of the $\fc_i$ will be the desired $f$-vanishing infinite additive subgroup of $\fa$.

Start by setting $\fc_1 = 0$, so $f$ vanishes on $\bU(S,\fc_1) = \id$ by assumption.  Assume
now that we have constructed an $f$-vanishing finite additive subgroup $\fc_i$ of $\fa$.
To construct an $f$-vanishing finite additive subgroup $\fc_{i+1}$ of $\fa$ with $\fc_i \subsetneq \fc_{i+1}$,
it is enough to find some nonzero $d \in \fa \setminus \fc_i$ such that $\fc_i + \Field_p d$ is $f$-vanishing.
To simplify our notation, we will let $\fc = \fc_i$.  Letting $\fd$ be an infinite additive subgroup
of $\fa$ with $\fc \cap \fd = 0$, we will find the desired nonzero $d$ in $\fd$.

For $u \in \fc$ and $c \in \bbF_p$ and $g \in S$, define a morphism of varieties
\[\gamma_{u,c,g} \colon \bbA^1 \to \bU, \qquad \gamma_{u,c,g}(t)=\presup{u+ct}{g}.\]
Here have crucially used the fact (Proposition~\ref{proposition:extendpositive}) that the $\bbG_m$ action extends to $\overline{\bbG}_m=\bbA^1$. 
Enumerate the finite set $\Set{$\gamma_{u,c,g}$}{$u \in \fc$, $c \in \Field_p$, $g \in S$}$ as
$\{\gamma_1,\ldots,\gamma_N\}$.  By definition, for $t \in k$ the group $\bU(S, \fc+\bbF_p t)$ is the subgroup
of $\bU(k)$ generated by $\Set{$\gamma_i(t)$}{$1 \leq i \leq N$}$.  Let $\cW_N \subset F_N$ be the
finite set provided by Corollary~\ref{corollary:words}, and for $w \in \cW_N$ define a morphism of varieties 
\[\phi_w \colon \bbA^1 \to \bU, \qquad \phi_w(t)=w(\gamma_1(t), \ldots, \gamma_N(t)).\]
It follows that for $t \in k$ we have 
\begin{equation}
\label{eqn:gamma}
\bU(S,\fc+\bbF_p t) = \Set{$\phi_w(t)$}{$w \in \cW_N$}.
\end{equation}
Since $\gamma_{u,c,g}(0) = \presup{u}{g}$ for all $u \in \fc$ and $c \in \bbF_p$ and $g \in S$, we also
have
\begin{equation}
\label{eqn:vanishinggamma}
\bU(S,\fc) = \Set{$\phi_w(0)$}{$w \in \cW_N$}.
\end{equation}
For $w \in \cW_N$, the function
\[f_w\colon \bbA^1(k) \rightarrow \Field_p, \qquad f_w(t)=f(\phi_w(t))\]
is an A-polynomial.  Since $\fc$ is $f$-vanishing, \eqref{eqn:vanishinggamma} implies that $f_w(0) = 0$.
Proposition~\ref{prop:apolynomial} thus implies that we can find some nonzero $d \in \fd$ such that
$f_w(d) = 0$ for all $w \in \cW_N$.  By \eqref{eqn:gamma}, this implies that $\fc+\Field_p d$ is
$f$-vanishing, as desired.
\end{proof}

\subsection{Replacement for Hall's theorem}

Our proof of Theorem~\ref{maintheorem:positiveideal} in characteristic~0 crucially relied Hall's theorem on abelian-by-nilpotent groups 
via Corollary~\ref{corollary:separatenil}. We now prove a result that will serve as a replacement for this in positive characteristic. 
For a group $G$ and a $\Z[G]$-module $M$, let $M_G$ denote the $G$-coinvariants of $M$, i.e., the largest quotient of $M$ on which 
$G$ acts trivially.  We begin with a lemma.

\begin{lemma}
\label{lem:abeliancoinv}
Let $p$ be a prime, let $\Lambda=\Z[1/p]$, and let $G$ be an abelian group of exponent $p$. Let
$M$ be a $\Lambda[G]$-module and let $m \in M$ be nonzero.  Then there exists a subgroup $H \subset G$ of index
$1$ or $p$ such that the image of $m$ in $M_{H}$ is nonzero.
\end{lemma}
\begin{proof}
Let $I \subset \Lambda[G]$ be the annihilator of $m$.  Since $m \neq 0$, this is a proper ideal, so it
is contained in a maximal ideal $J$.  Let $\mu_p$ be the group of $p^{\text{th}}$ roots of unity in the field
$\Lambda[G]/J$.  It is possible for there to be no nontrivial $p^{\text{th}}$ roots of unity in this field, in which case
$\mu_p = 1$.  The map $\Lambda[G] \rightarrow \Lambda[G]/J$ induces a group homomorphism $G \rightarrow \mu_p$; let
$H$ be its kernel, which has index $1$ or $p$ in $G$.  Since $H$ acts trivially on $\Lambda[G]/J \neq 0$ and
the map $\Lambda[G]/I \rightarrow \Lambda[G]/J$ is surjective, we deduce that $(\Lambda[G]/I)_H \neq 0$.

Since $I$ is the annihilator of $m$, there is an injection $\Lambda[G]/I \rightarrow M$ whose image is
the $\Lambda[G]$-span of $m$.  This induces a map $(\Lambda[G]/I)_H \rightarrow M_H$ whose image
is the $\Lambda[G/H]$-span of $m$ in $M_H$.  To prove that the image of $m$ in $M_H$ is nonzero, it
is enough to show that the map $(\Lambda[G]/I)_H \rightarrow M_H$ is injective.

In fact, we claim that taking $H$-coinvariants is an exact functor on the category of $\Lambda[H]$-modules. First suppose $H$ is finite. Then it is well-known that for a
$\Lambda[H]$-module $N$, the group homology $\HH_k(H;N)$ for $k>0$ is annihilated by the order of $H$, which is a power of $p$. Since $N$ is a $\Lambda$-module, multiplication by $p$ is an isomorphism on $N$. It follows that $\HH_k(H;N)=0$ for $k>0$. This implies the claim in this case, as group homology is the derived functor of coinvariants. We now treat the general case. Write $H=\bigcup_{i \in I} H_i$ where the $H_i$ are finite subgroups of $H$. Then $M_H=\varinjlim M_{H_i}$. Since both
direct limits and the formation of $H_i$-coinvariants are exact, the claim follows.
\end{proof}

The following proposition is our substitute for Corollary~\ref{corollary:separatenil} in positive characteristic:

\begin{proposition}
\label{prop:positiveidealcorep}
Let $\bU$ be a smooth connected unipotent group over an infinite field $k$ of positive characteristic $p$
equipped with a positive $\bbG_m$-action.  Set $\Lambda = \Z[1/p]$.  Let $S$ be a finite subset of $\bU(k)$,
let $\fa$ be an infinite additive subgroup of $k$, let $M$ be a $\Lambda[\bU(S,\fa)]$-module, and
let $m \in M$ be nonzero.  Then there exists an infinite additive subgroup $\fc$ of $\fa$ such
that the image of $m$ in $M_{\bU(S,\fc)}$ is nonzero.
\end{proposition}
\begin{proof}
We will prove this by induction on $n = \dim(\bU)$.  The case $n=0$ being trivial, assume that $n>0$ and
that the lemma is true in smaller dimensions.  Proposition \ref{proposition:centerpositive} says 
there exists a $\bbG_m$-stable central subgroup $\bA \lhd \bU$ with $\bA \cong \bbG_a$.  Since $\Char(k) = p$, the
abelian group underlying $k$ has exponent $p$.  It follows that the intersection
\[\bA(k) \cap \bU(S,\fa) \subset \bA(k) = k\]
is an abelian group of exponent $p$.  By Lemma~\ref{lem:abeliancoinv}, there is a subgroup
$H \subset \bA(k) \cap \bU(S,\fa)$ of index either $1$ or $p$ such that the image of $m$ in $M_H$ is nonzero.
Choose an additive homomorphism $\lambda\colon \bA(k) \rightarrow \Field_p$ such that
$\ker(\lambda) \cap \bU(S,\fa) = H$.  

Since $\bA$ is (trivially) a split unipotent group, we can apply Proposition~\ref{proposition:unipotentproduct} to $(\bG,\bV) = (\bU,\bA)$.  We
deduce that there
is a subvariety $\bX$ of $\bU$ containing the identity such that the multiplication map $\bA \times \bX \to \bU$ is an isomorphism of varieties. Using
this product structure, let $\pi\colon \bU \rightarrow \bA$ be the projection onto the first factor. Since $\bX$ contains the identity, it follows that $\pi \vert_{\bA}$ is the identity map.
Define $f\colon \bU(k) \rightarrow \Field_p$ to be the composition
\[\bU(k) \stackrel{\pi}{\longrightarrow} \bA(k) \stackrel{\lambda}{\longrightarrow} \Field_p.\]
The map $f$ is an A-polynomial, so by Proposition~\ref{prop:groupapolynomial} there exists an infinite
additive subgroup $\fb$ of $\fa$ such that $f$ vanishes on $\bU(S,\fb)$.  This implies
that $\bA(k) \cap \bU(S,\fb) \subset H$, so the image of $m$ in
$M_{\bA(k) \cap \bU(S,\fb)}$ is nonzero.  Let $\overline{\bU} = \bU/\bA$ and let 
$\oS \subset \overline{\bU}(k)$ be the image of $S \subset \bU(k)$.  The action
of $\bU(S,\fb)$ on $M_{\bA(k) \cap \bU(S,\fb)}$ factors through $\overline{\bU}(\oS,\fb)$.
Using our inductive hypothesis, we can find an infinite additive subgroup $\fc$ of $\fb$
such that the image of $m$ in
\[\left(M_{\bA(k) \cap \bU(S,\fb)}\right)_{\overline{\bU}(\oS,\fc)}\]
is nonzero.  This implies that the image of $m$ in $M_{\bU(S,\fc)}$ is nonzero, as desired.
\end{proof}

\subsection{Modules over unipotent groups in char p}

The following is our generalization of Lemma \ref{lemma:polylemma} to the setting of unipotent groups over fields of positive characteristic.

\begin{proposition} 
\label{proposition:polylemmacharp}
Let $\bU$ be a smooth connected unipotent group over a field $k$ of positive characteristic~$p$ equipped with a positive $\bbG_m$-action. Set $\Lambda=\Z[1/p]$. Let $M$ be a $\Lambda[\bU(k)]$-module and let $N \subset M$ be a submodule. For some $m_1,\ldots,m_n \in M$ and $g_1,\ldots,g_n \in \bU(k)$, assume that
\[\text{$\presup{t}{g_1} \cdot m_1 + \cdots + \presup{t}{g_n} \cdot m_n \in N$ for all $t \in \bbG_m(k)$}.\]
Then $m_1+\cdots+m_n \in N$.
\end{proposition}
\begin{proof}
Replacing $M$ by $M/N$, we can assume that $N=0$. Let $S=\{g_1,\ldots,g_n\}$, let $\fb$ be a non-zero additive subgroup of $k$, and let $t \in \fb$ be non-zero.  Letting
$\equiv$ denote equality in the $\bU(S,\fb)$-coinvariants of $M$, since the elements $\presup{t}{g_i} \in \bU(S,\fb)$ act trivially on these coinvariants we have
\[0 = \presup{t}{g_1} \cdot m_1 + \cdots + \presup{t}{g_n} \cdot m_n \equiv m_1 + \cdots + m_n.\]
We thus see that $m_1+\cdots+m_n$ maps to~0 in $M_{\bU(S,\fb)}$ for all non-zero $\fb$. Proposition~\ref{prop:positiveidealcorep} (applied with $\fa=k$) implies that $m_1+\cdots+m_n=0$.
\end{proof}

\subsection{Conclusion}
\label{section:assemblecharp}

The following is Theorem \ref{maintheorem:positiveideal} in the special case where $\Char(k)$ is positive and $\Char(\Field) \neq \Char(k)$.

\begin{theorem}
\label{theorem:positiveidealcharp}
Let $\bU$ be a smooth connected unipotent group over a field $k$ of positive characteristic~$p$ equipped with a positive action of $\bbG_m$ and let
$\bbF$ be another field with $\Char(\bbF) \neq p$.  Let $I \subset \bbF[\bU(k)]$ be a left ideal that is stable under $\bbG_m$ and not contained in the augmentation
ideal.  Then $I = \bbF[\bU(k)]$.
\end{theorem}
\begin{proof}
The proof is nearly identical to that of Theorem~\ref{theorem:positiveidealchar0}. 
For $g \in \bU(k)$, write $[g]$ for the associated element of $\bbF[\bU(k)]$.  Let $\epsilon\colon \bbF[\bU(k)] \rightarrow \bbF$ be
the augmentation.  Since $I$ is not contained in the augmentation ideal, there
exists some $x \in I$ with $\epsilon(x) = 1$.  Write this as
\[x = \sum_{i=1}^n c_i [g_i] \in I \quad \text{with $g_1,\ldots,g_n \in \bU(k)$, $c_1,\ldots,c_n \in \bbF$, and $\sum_{i=1}^n c_i = 1$}.\]
Since $I$ is stable under the action of $\bbG_m(k)$, for all $t \in \bbG_m(k)$ we have
$\presup{t}{x} \in I$, so
\[\sum_{i=1}^n c_i [\presup{t}{g_i}] = \sum_{i=1}^n \presup{t}{g_i} \cdot c_i [1] \in I.\]
Since $\Char(\Field) \neq p$, the field $\Field$ is an algebra over $\Lambda = \Z[1/p]$, so we can regard $\bbF[\bU(k)]$ as
a $\Lambda[\bU(k)]$-module.  Applying Proposition \ref{proposition:polylemmacharp} with $M=\bbF[\bU(k)]$ and $N=I$, we deduce that
$[1] = \sum_i c_i [1] \in I$,
so $I = \bbF[\bU(k)]$.
\end{proof}

\end{document}